\documentclass[10pt,reqno]{amsart}
\usepackage {amssymb}
\usepackage {amsmath}
\usepackage {bbm}
\usepackage{amsthm}
\usepackage{mathtools}
\usepackage{graphicx}
\usepackage {amscd}
\usepackage {epic}
\usepackage[mathscr]{euscript}

\usepackage[dvipsnames]{xcolor}
\usepackage[colorlinks,citecolor=OliveGreen,linkcolor=Mahogany,urlcolor=Plum,pagebackref]{hyperref}
\usepackage[alphabetic]{amsrefs}

\usepackage{enumerate}
\usepackage{tikz}
\usepackage{tikz-cd}
\usepackage{verbatim,color,geometry}
\usepackage[all]{xy}
\usepackage{enumitem}
\usepackage{bm}
\newcommand{\A}{\mathbb{A}}
\newcommand{\G}{\mathbb{G}}
\newcommand{\K}{\mathbb{K}}
\newcommand{\N}{\mathbb{N}}
\newcommand{\PP}{\mathbb{P}}
\newcommand{\Q}{\mathbb{Q}}
\newcommand{\R}{\mathbb{R}}
\newcommand{\T}{\mathbb{T}}
\newcommand{\Z}{\mathbb{Z}}

\newcommand{\cL}{{\mathcal{L}}}

\newcommand{\cF}{{\mathcal{F}}}
\newcommand{\cM}{{\mathcal{M}}}

\newcommand{\cX}{{\mathcal{X}}}

\newcommand{\fX}{{\mathfrak{X}}}
\newcommand{\fL}{{\mathfrak{L}}}

\newcommand{\fD}{{\mathfrak{D}}}

\newcommand{\bH}{{\mathbf{H}}}
\newcommand{\bP}{{\mathbf{P}}}

\newcommand{\Aut}[0]{\operatorname{Aut}}

\newcommand{\Supp}{{\rm Supp}}
\newcommand{\Hom}{{\rm Hom}}

\newcommand{\ord}{{\rm ord}}
\newcommand{\gr}{{\rm gr}}

\numberwithin{equation}{section}
\newtheorem{thm}{Theorem}[section]

\newtheorem{lem}[thm]{Lemma}
\newtheorem{cor}[thm]{Corollary}

\newtheorem{prop}[thm]{Proposition}

\newtheorem{conj}[thm]{Conjecture}

\theoremstyle{definition}
\newtheorem{defn}[thm]{Definition}

\newtheorem{example}[thm]{Example}

\newtheorem{rem}[thm]{Remark}

\newtheorem{defn-thm}[thm]{Definition--Theorem}  
\newtheorem{defn-prop}[thm]{Definition--Proposition}  
\newtheorem{defn-lem}[thm]{Definition--Lemma}  

\theoremstyle{remark}

\newcommand{\Fut}{{\rm Fut}}

\newcommand{\vol}{{\rm vol}}

\newcommand{\cD}{{\mathcal{D}}}
\newcommand{\cG}{{\mathcal{G}}}
\newcommand{\cO}{{\mathcal{O}}}

\newcommand{\bQ}{{\mathbb{Q}}}

\newcommand{\Spec}{{\rm Spec}}

\newcommand{\la}{\lambda}

\newcommand{\bmu}{\bm{\mu}}

\newcommand{\Gal}{{ \rm Gal}}
\newcommand{\PGL}{{ \rm PGL}}
\newcommand{\Map}{{ \rm Map}}
\newcommand{\Hilb}{{ \rm Hilb}}
\newcommand{\lnorm}[1]{\lVert#1\rVert_2}
\newcommand{\mnorm}[1]{\lVert#1\rVert_{\rm m}}

\newcommand{\bR}{\mathbb{R}}
\DeclareMathOperator{\Filt}{Filt}

\newcommand{\HL}[1]{{\textcolor{red}{[Daniel: #1]}}}
\newcommand{\CX}[1]{{\textcolor{red}{[Chenyang: #1]}}}

\newcommand{\sX}{\mathscr{X}}
\newcommand{\sS}{\mathscr{S}}
\newcommand{\uMap}{\underline{\Map}}
\newcommand{\bfm}{\mathbf{m}}
\newcommand{\prim}{\mathrm{prim}}

\title[Properness of K-moduli and optimal degenerations]{On properness of K-moduli spaces and optimal degenerations of Fano varieties}

\subjclass[2010]{14D20, 14J45}
\keywords{Fano varieties, K-stability, moduli}

\author{Harold Blum}
\address{Department of Mathematics, Stony Brook University, Stony Brook, NY 11794, USA}
\email{harold.blum@stonybrook.edu}

\author{Daniel Halpern-Leistner}
\address{Department of Mathematics, Malott Hall, Cornell University, Ithaca, NY 14853, USA}
\email{daniel.hl@cornell.edu}

\author{Yuchen Liu}
\address{Department of Mathematics, Northwestern University, Evanston, IL 60208, USA}
\email{yuchenl@northwestern.edu}

\author{Chenyang Xu}
\address{Department of Mathematics, Princeton University, Princeton, NJ 08544, USA}
\email     {chenyang@princeton.edu}
\address   {Department of Mathematics, MIT, Cambridge, MA 02139, USA}
\email     {cyxu@math.mit.edu}
\address   {Beijing International Center for Mathematical Research,       Beijing 100871, China}
\email     {cyxu@math.pku.edu.cn}

\begin{document}

\pagenumbering{arabic}

\begin{abstract}
We establish an algebraic approach to prove the properness of moduli spaces of K-polystable Fano varieties and reduce the problem to a conjecture on destabilizations of K-unstable Fano varieties.
Specifically, we prove that if the stability threshold of every K-unstable Fano variety is computed by a divisorial valuation, then such K-moduli spaces are proper. 
The argument relies on studying certain optimal destabilizing test configurations and constructing a $\Theta$-stratification on the moduli stack of Fano varieties.
\end{abstract}

\maketitle

\section{Introduction}

A key feature of the K-stability theory for Fano varieties, which was probably once beyond the imagination of algebraic geometers, is that it yields a good theory for moduli spaces. There has been significant progress in the recent years to give a purely algebro-geometric construction of such moduli spaces, called {\it K-moduli spaces}. Specifically, it has been shown that there is a finite type Artin stack  $\cM_{n,V}^{\rm Kss}$  parametrizing families of K-semistable Fano varieties of dimension $n$ and volume $V$ and the stack  admits a  morphism to a separated good moduli space $M_{n,V}^{\rm Kps}$ whose points parameterize those that are K-polystable. See \cites{Jia17, BX19, ABHLX19, BLX19, Xu20,  XZ20}. 

The remaining major challenge is to show that the moduli space $M_{n,V}^{\rm Kps}$ is proper. This is known for the component which parametrizes $\Q$-Gorenstein smoothable K-polystable Fano varieties (see \cite{LWX19}), whose proof essentially relies on analytic results in \cite{DS14, TW19}. In this note, we aim to establish an algebraic strategy to prove the properness. More precisely, we will show that it is implied by the existence of certain optimal destabilizing degenerations. 

We will follow a strategy sometimes called Langton's algorithm. Historically, Langton proved the valuative criterion of properness for the moduli space of polystable sheaves on a smooth projective variety $X$ of arbitrary dimension (see \cite{Lan75}). Starting with a semistable sheaf $F_K$ on $X\times {\rm Spec}(K)$, where $K$ is the fraction field of a DVR $R$, Langton shows that one can modify any degeneration $F_\kappa$ of $F_K$ on $X\times  {\rm Spec}(\kappa)$, where $\kappa$ is the residue field of $R$, with a sequence of uniquely determined elementary transformations, so that the `instability' of $F_\kappa$ decreases. Moreover, he showed that after finitely many steps this process terminates with the degeneration becoming semistable. This method was 
adapted to $G$-bundles on curves in \cite{Hein08} and finally abstracted in \cite{AHLH18}, where it is shown that Langton's algorithm can be carried out on an Artin stack as long as it admits a {\it $\Theta$-Stratification} (see Definition \ref{d:ThetaStrat}).

The notion of a $\Theta$-stratification was defined in \cite{HL14} to generalize the Harder-Narasimhan stratification of the moduli of coherent sheaves on a projective scheme as well as the Kempf-Ness stratification in GIT.
The definition is based on the idea that 
(i) the stability of a point $x$ in a stack $\sX$ is determined by considering maps $f:[\A^1/\G_m] \to \sX$ such that $f(1)=x$, (ii) if $x$ is unstable, then there should be a unique optimal destabilizing map, and (iii) these optimal destabilizations should satisfy certain properties in families and can be used to stratify the unstable locus of the stack.
When $\sX$ is the moduli stack of coherent sheaves on a projective scheme and $E$ a coherent sheaf on that scheme, 
maps $[\A^1/\G_m] \to \sX$ sending $1\to [E]$ are in bijection with 
filtrations of $E$, and when $E$ is unstable, the optimal destabilization is given by the  Harder-Narasimhan filtration.

In this paper, we seek to define a $\Theta$-stratification on the stack
$\cM^{\rm Fano}_{n,V}$ parametrizing  families of $\Q$-Fano varieties of dimension $n$ and volume $V$.
Since maps $[\A^1/\G_m] \to \cM^{\rm Fano}_{n,V}$  such that $f(1) =[X]$ are equivalent to special test configurations of $X$, we must identify a unique optimal destabilizing test configuration for each K-unstable Fano variety $X$.

A natural starting point is to normalize the Futaki invariant of a test configuration by a norm and hope to find a unique test configuration minimizing the invariant.
It is expected that if we normalize by the minimum norm, then the infimum
\begin{equation}\label{E:infimum}
\displaystyle \inf_{\cX} \frac{\Fut(\cX)}{ \mnorm {\cX} }\end{equation} that runs through all non-trivial special test configurations $\cX$ of $X$ is attainable. 
This infimum is closely related to the stability threshold of $X$, denoted $\delta(X)$, defined and studied in \cite{FO18,BJ20}. Indeed, by \cite{BLZ19}, the previous infimum equals $\delta(X)-1$ and is achieved if and only if a well known conjecture (see Conjecture \ref{c:optdest})  regarding valuations computing $\delta(X)$ holds. 

However, unlike the optimal destabilizing degeneration in Geometric Invariant Theory \cite{Kem78}, a degeneration achieving the infimum \eqref{E:infimum} need not be unique. Thus, only considering $ \frac{\Fut(\cX)}{ \mnorm {\cX} }$ is not enough to define a $\Theta$-stratification. To remedy this, we consider the set of test configurations achieving the above infimum and  minimize the Futaki invariant normalized by the $L^2$ norm among these test configurations. This amounts to minimizing the bi-valued invariant 
$\Big( \frac{\Fut(\cX)}{ \mnorm {\cX} }, \frac{\Fut(\cX  )}{ \lnorm{\cX}}\Big)$
with respect to the lexicographic order on $(\R\cup\{\pm\infty\})^2$.
Assuming Conjecture \ref{c:optdest} holds, we verify that there exists a unique special test configuration minimizing this function and that these minimizing test configurations define a {\it well-ordered} $\Theta$-stratification on $\cM_{n,V}^{\rm Fano}$ (see Definition \ref{d:ThetaStrat}).

\subsection{Main results}\label{S:main_results}
Below, we will give a more detailed description of our main results, which are stated in the more general setting of log Fano pairs, rather than $\Q$-Fano varieties.

 For a special test configuration $(\cX,\cD)$ of a log Fano pair $(X,D)$, we set
\[
\bmu \left( \cX,\cD \right) 
:=
\left( \frac{\Fut(\cX , \cD )}{ \mnorm {\cX,\cD} }, \frac{\Fut(\cX , \cD )}{ \lnorm{\cX,\cD}}\right) \in \R^2.
\]
When $(X,D)$ is K-unstable, we set
\begin{equation}\label{eq:M=infmu}
M^{\bmu} (X,D) : = \inf_{ (\cX,\cD)} \bmu(\cX,\cD)
,
\end{equation}
where the infimum runs through special test configurations of $(X,D)$.
Note that by  \cite{BLZ19},
$\inf_{ (\cX,\cD)}  \frac{ \Fut( \cX,\cD) }{ \mnorm{\cX,\cD}}$ is equal to $\delta(X,D)-1$ and the infimum is a minimum if the following conjecture holds.

\begin{conj}[Optimal Destabilization]\label{c:optdest}
If $(X,D)$ is a K-unstable log Fano pair, then there exists a divisor $E$ over $X$ computing the infimum 
$$
\delta(X,D):=\inf_{E} \frac{A_{X,D}(E)}{S_{X,D}(E)}
.$$
\end{conj}

While the above conjecture remains open, it is known that there exists a  quasi-monomial valuation, rather than divisorial valuation, achieving the infimum (see \cite{BJ20,BLX19,Xu20}).
By \cite{Xu20b}, Conjecture \ref{c:optdest} would hold
if one knew the finite generation of the associated graded ring induced by a quasi-monomial valuation which computes $\delta(X,D)$ on the ring $\bigoplus_{m\in \mathbb N} H^0(-mr(K_X+\Delta))$ for a sufficiently divisible positive integer $r$.

We prove the following result on test configurations achieving the infimum in  \eqref{eq:M=infmu}.

\begin{thm}\label{t:HNfilts}
Let $(X,D)$ be a log Fano pair that is K-unstable. 
\begin{itemize}
\item[(1)] (Existence)  The  pair  $(X,D)$ satisfies the conclusion of Conjecture \ref{c:optdest} if and only if there exists a special test configuration achieving the infimum in  \eqref{eq:M=infmu}.
  \item[(2)] (Uniqueness)  Any two special test configurations of $(X,D)$ achieving the infimum in \eqref{eq:M=infmu} are isomorphic after (possible) rescaling.
\end{itemize}
\end{thm}

To prove Theorem \ref{t:HNfilts}, we use that any two test configurations which minimize $\frac{ \Fut}{\mnorm{\,\,}}$ can be connected by an equivariant family over $\mathbb{A}^2$ following \cite{BLZ19}. With this in hand, we can analyze $\frac{ \Fut}{\lnorm{\,\,}}$ and conclude the uniqueness result. 

By analyzing the properties of these optimal destabilizing test configurations in families, we strengthen the theorem above by showing that $\mu$ determines a $\Theta$-stratification  on  $\cM_{n,V,c}^{\rm Fano}$, which denotes the moduli stack  parametrizing families of log Fano pairs with fixed numerical invariants (see Definition \ref{d-logfanofamily}). We give a direct construction of this $\Theta$-stratification, and we also discuss how this follows from the general theory of $\Theta$-stability developed in \cite{HL14}.

\begin{thm}\label{t-theta}
If Conjecture \ref{c:optdest} holds,
then $\bm{\mu}$ determines a $\Theta$-stratification on $\cM_{n,V,c}^{\rm Fano}$. 
\end{thm}

Now, we consider the open subfunctor $\cM_{n,V,c}^{\rm Kss}\subset \cM_{n,V,c}^{\rm Fano}$ parametrizing families that have K-semistable fibers, where the openness was shown in \cite{BLX19,Xu20}.
By \cite{BX19, ABHLX19}, we know $\cM_{n,V,c}^{\rm Kss}$ admits a separated good moduli space $M_{n,V,c}^{\rm Kps}$ (as an algebraic space). As a corollary of Theorem \ref{t-theta} and \cite{AHLH18}, we conclude the following.

\begin{cor}\label{cor:properness}
If Conjecture \ref{c:optdest} holds,
then $M_{n,V,c}^{\rm Kps}$ is proper. 
\end{cor}

Note that in \cite{XZ19}, building on \cite{CP21}, it is shown that the CM line bundle on $M_{n,V,c}^{\rm Kps}$ is ample, provided that $M_{n,V,c}^{\rm Kps}$  is proper and its points  parameterize reduced uniformly K-stable log Fano pairs. The latter would follow from a conjecture similar to Conjecture \ref{c:optdest} (see \cite[Conjecture A.12]{XZ19}).

\bigskip

The paper is organized as follows. After providing  background in \S\ref{s:prelim}, we collect information on properties of $\bmu$ when restricted to 1-parameter subgroups of a torus acting on a log Fano pair in \S\ref{s:logFanotorus}.
We then prove Theorem \ref{t:HNfilts} on the existence and uniqueness of minimizers of $\bmu$ in \S\ref{s:existance} and \S\ref{s:uniqueness}.
In \S\ref{s:lsc}, we analyze the behavior  of $M^{\bmu}$  in families. 
Lastly,  we prove Theorem \ref{t-theta} and Corollary \ref{cor:properness} in  \S\ref{s:theta}, and discuss an alternative approach using the general framework of $\Theta$-stability in \S\ref{s:general_stability}.

\bigskip

\emph{Postscript remarks.} After the first version of this article was posted on the arXiv, in \cite{LXZ21}, the third, fourth authors and Zhuang prove that any valuation computing $\delta(X,\Delta)< \frac{n+1}{n}$ where $n=\dim(X)$ has a finitely generated associated graded ring. This confirms Conjecture \ref{c:optdest} in full generality and hence, combined with Theorem \ref{t-theta} and Corollary \ref{cor:properness}, leads to proofs of the existence of a $\Theta$-stratification on the  stack $\cM_{n,V,c}^{\rm Fano}$ and  the properness of the K-moduli space $M_{n,V,c}^{\rm Kps}$.

\bigskip

\noindent {\bf Acknowledgement}: \textit{The authors thank Jarod Alper and Jochen Heinloth for helpful conversations.
HB was partially supported by NSF grant DMS-1803102.
DHL was partially supported by a Simons Foundation Collaboration grant and NSF CAREER grant DMS-1945478.
YL was partially supported by NSF grant DMS-2001317.
CX was partially supported by NSF grant DMS-1901849.
}

\section{Preliminaries}\label{s:prelim}

\subsection{Conventions} 
Throughout, we work over an algebraically closed  characteristic 0 field $k$.
We  follow standard  terminologies in \cite{KM98, Kol13}.

A \emph{pair} $(X,D)$ is composed of a normal variety $X$ and an effective $\Q$-divisor $D$
on $X$ such that $K_X+D$ is $\mathbb{Q}$-Cartier. 
See \cite[2.34]{KM98} for the definitions of \emph{klt} and \emph{lc} pairs.
A pair $(X,D)$ is \emph{log Fano} if $X$ is projective, $(X,D)$ is klt, and $-K_X-D$ is ample. 
A variety $X$ is {\it $\mathbb{Q}$-Fano} if $(X,0)$ is log Fano.

\begin{defn}\label{d:familynormalbase}
A \emph{family of log Fano pairs} $f:(X,D) \to T$  over a normal scheme $T$ is the data of a flat surjective morphism of schemes $f:X \to T$ and a
$\Q$-divisor $D$ on $X$ satisfying
\begin{enumerate}
\item  $T$ is normal and $f$ has normal fibers (hence, $X$ is normal as well),
\item $\Supp(D)$ does not contain a fiber,
\item $K_{X/T} +D$ is $\Q$-Cartier, and
\item $(X_{\overline{t}},D_{\overline{t}})$ is  a log Fano pair for each $t\in T$.
\end{enumerate}
Using \cite{Kol19}, this definition can be extended to the case when $T$ is not-necessarily normal; see Definition \ref{d-logfanofamily}.
\end{defn}

\subsection{K-stability}

\subsubsection{Definition}
In this section, we recall the definition of K-stability \cite{Tia97,Don02}.  Following \cite{LX14}, we define these notions only using special test configurations.

\begin{defn}
Let $(X,D)$ be a log Fano pair. A \emph{special test configuration} $(\cX,\cD)$ of  $(X,D)$ is the data of 
\begin{enumerate}
    \item a family of log Fano pairs $(\cX,\cD)\to \A^1$,
    \item a $\G_m$-action on $(\cX,\cD)$ extending the standard action on $\A^1$, and
    \item an isomorphism $(\cX_1,\cD_1)\simeq (X,D)$.
 \end{enumerate}
The test configuration is a \emph{product} if $(\cX,\cD)\simeq (X,D)\times \A^1$ as a family of log Fano pairs  and  \emph{trivial} if the latter isomorphism is $\G_m$-equivariant with respect to the trivial action on $(X,D)$ and the standard action on $\A^1$.
\end{defn}

A special test configuration $(\cX,\cD)$ can be \emph{scaled} by a positive integer $d$.
Indeed, the base change of $(\cX,\cD)$ by the map $\A^1 \to \A^1$ sending $t\mapsto t^d$ is a special test configuration and we denote it by $(\cX^{(d)},\cD^{(d)})$. We call a special test configuration $(\cX,\cD)$ \emph{primitive} if it is not a scaling of some other test configuration with $d\geq 2$.

A special test configuration $(\cX,\cD)$ has a natural compactification
$(\overline{\cX},\overline{\cD})\to \PP^1$ constructed by gluing
 $\cX$  and $X\times( \PP^1 \setminus 0) $  along their respective open sets $\cX\setminus \cX_0$ and $X\times (\A^1\setminus \{0\})$.
The \emph{generalized Futaki invariant} of a special test configuration $(\cX,\cD)$ is defined by
\[
\Fut(\cX,\cD)  :  =- 
\frac{ (-K_{\overline{\cX}/\PP^1} - \overline{\cD})^{n+1}}
{(n+1) (-K_X-D)^n}, 
\]
where $n$ is the dimension of $X$. Note that this definition using the intersection formula is equivalent to the original definition from \cite{Tia97, Don02} by \cite{Wan12, Oda13b}.

\begin{defn}[K-stability]
\cite{Tia97,Don02,LX14}
A log Fano pair $(X,D)$ is 
\begin{enumerate}
    \item \emph{K-semistable} if $\Fut(\cX,\cD)\geq 0$ for all special test configurations of $(\cX,\cD)$;
    \item \emph{K-polystable} if it is K-semistable and any special test configuration $(\cX,\cD)$ such that $\Fut(\cX,\cD)=0$ is a product.
\end{enumerate}
We call $(X,D)$ \emph{K-unstable} if it is not K-semistable.
\end{defn}

\subsection{Further invariants of test configurations}\label{ss:furtherinvariants}

Let $(X,D)$ be a log Fano pair with a $\G_m$-action denoted by $\xi$. Fix $r\in \Z_{>0}$ such that $L:=-r(K_X+D)$ is a Cartier divisor and 
set 
\[
R:= \bigoplus_{m \in \N} R_{m} := \bigoplus_{m \in \N} H^0(X,\cO_X(L)).
\]
Note that the $\G_m$-action $\xi$ on $X$ induces a canonical $\G_m$-action on $R_m$ and we write $\bigoplus_{\lambda\in \Z} R_{m,\lambda}$ for the weight decomposition.  
Set
\[
N_m := \dim R_m, 
\quad
w_m := \sum_{\lambda \in \Z } \lambda \dim R_{m,\la},
\quad 
\text{ and }
\quad
q_m: = 
\sum_{\lambda \in \Z } \lambda^2 \dim R_{m,\lambda}
.\]
By general theory (see  \cite{Don05,BHJ17}), for $m\gg0$, there are Laurent expansions with rational coefficients such that
\[
\frac{w_m}{mrN_m}= F_0 + F_1 m^{-1} + F_2 m^{-2} + \cdots 
\]
\[
\frac{q_m}{(mr)^2 N_m} = Q_0 + Q_1 m^{-1}+ Q_2 m^{-2}+ \cdots 
.\]
The 
\emph{Futaki invariant} \cite{Don02}
 of $(X,D;\xi)$ is given by
\[
\Fut(X,D;\xi) := -F_0 
.\]
The $L^2$ \emph{norm} \cite{Don05} and \emph{minimum norm}
\cite{Der16} are defined by 
\[
\mnorm{X,D;\xi} : = F_0  - \lambda _{\min}
\quad 
\text{ and } \quad 
\lnorm{X,D;\xi} := \sqrt{Q_0 - F_0^2}
,\]
where  $\lambda_{\min} := \inf_m  \frac{\lambda_{\min,m}}{mr}$ and $\lambda_{\min,m} : = \min \{ \lambda\in \Z \, \vert \,  R_{m,\lambda} \neq 0\}$.

For a special  test configuration
$(\cX,\cD)$ of $(X,D)$, write $\xi$ for the induced $\G_m$-action on $(\cX_0,\cD_0)$. 
As observed in \cite{Wan12,Oda13b}, 
$\Fut{(\cX,\cD)} = \Fut(\cX_0,\cD_0;\xi)$.
The 
 \emph{minimum norm}
and 
$L^2$ \emph{norm}
 of $(\cX,\cD)$
are defined by 
\[
\mnorm{\cX,\cD}: = \mnorm{\cX_0,\cD_0;\xi}
\quad \text{ and }\quad
\lnorm{\cX,\cD}: = \lnorm{\cX_0,\cD_0;\xi}.
\]
These norms are non-negative and equal zero if and only if $(\cX,\cD)$ is trivial; see \cite[Corollary B]{BHJ17} and \cite[Theorem 1.3]{Der16}. We note that the \emph{minimum norm} agrees with the non-Archimedian $I-J$ functional in \cite{BHJ17} and also has an interpretation in terms of intersection numbers \cite[Remark 3.11]{Der16}  \cite[Remark 7.12]{BHJ17}.

More generally, for $p\in [1,+\infty]$, one can define the $L^p$ norm of a test configuration \cite{Don05} (see also \cite{His16} and \cite{BHJ17}*{Definition 6.5}).
Since  the minimum norm is equivalent to the $L^1$ norm by \cite{BHJ17}*{Remark 7.12} and the $L^1$ norm is less than or equal to   the $L^2$ norm, there exists  $c>0$ such that
\begin{equation}\label{e:mnorm<clnorm}
\mnorm{\mathcal{X},\mathcal{D}} \leq c \lnorm{\mathcal{X},\mathcal{D}}
\end{equation}
for all test configurations $(\mathcal{X},\mathcal{D})$ of $(X,D)$.

We will frequently use the following elementary fact  on the behavior of these invariants in families. 

\begin{lem}\label{l:futnormfiberwise}
Let $f:(X,D)\to T$ be a family of log Fano pairs  admitting a fiberwise $\G_m$-action $\xi$. If $T$ is connected, then $\Fut(X_t,D_t; \xi)$, $\mnorm{X_t,D_t; \xi}$, and $\lnorm{X_t,D_t; \xi}$ are independent of $t\in T$. 
\end{lem}

\begin{proof}
Fix a positive integer $r$ such that $L:=-r(K_{X/T}+D)$ is a Cartier divisor. 
Since $H^i(X_t,\cO_{X_t}(mL_t))=0$ for all $m,i>0$  and $t\in T$
by Kawamata-Viehweg vanishing,  $ f_*\cO_{X}(mL)$ is a vector bundle and commutes with base change.
Since $\xi$ induces a fiberwise $\G_m$-action on $ f_*\cO_{X}(mL)$, the vector bundle  admits a  direct sum decomposition into weight spaces 
$ f_*\cO_{X}(mL) =\bigoplus_{\la \in \Z} (f_*\cO_{X}(mL))_\la$, where each $(f_*\cO_{X}(mL))_\la$ is a vector bundle and commutes with base change.
Therefore,  $\dim (H^0(X_t, \cO_{X_t}(mL_t))_{\lambda})$ is independent of $t\in T$ and the result follows. 
\end{proof}

\subsection{Stability Threshold}
Here, we recall the stability threshold of a log Fano pair and its relation to the invariants in Section \ref{ss:bmu}.

\subsubsection{Definition}
Let $(X,D)$ be a log Fano pair of dimension $n$. A \emph{divisor over}  $X$ is the data of a prime divisor $E$ on a normal variety $Y$ with a proper birational morphism $\pi:Y\to X$.
The \emph{log discrepancy} of $E$
is defined by $A_{X,D}(E) : = {\rm coeff}_E(D_Y)+1$, where $K_Y+ D_Y = \pi^* (K_X+D)$.
The \emph{average order of vanishing} of  $-K_X-D$ along $E$
is given by 
\[
S_{X,D}(E) : = \frac{1}{(-K_X-D)^n} \int_0^\infty \vol( -\pi^*(K_X+D) - t E) \, dt
.\]

The \emph{stability threshold} (also known as the \emph{$\delta$-invariant}) of $(X,D)$
is given by 
\begin{equation}\label{eq:deltadef}
\delta(X,D) := \inf_{E} \frac{ A_{X,D}(E)}{S_{X,D}(E)}
\end{equation}
If a divisor $E$ over $X$ achieves the infimum, we say that $E$ \emph{computes} $\delta(X,D)$.

This invariant was originally introduced in \cite{FO18} using a different definition, which is equivalent to the latter by \cite{BJ20}.
It follows from \cite{Fuj19,Li17}, that $\delta(X,D)\geq 1$ if and only if $(X,D)$ is K-semistable. 

\subsubsection{Optimal Destabilization Conjecture}
Conjecture \ref{c:optdest} predicts that on a K-unstable log Fano pair $(X,D)$, the infimum in \eqref{eq:deltadef} is achieved.
Partial answers to this conjecture are known.
\begin{itemize}
    \item  When $X$ is smooth and $D=0$, the statement follows from \cite{DS16,RS19}, which rely on deep analytic results; see \cite[Theorem 6.7]{BLZ19}.
    \item For any K-unstable log Fano pair, a weaker version of the conjecture holds in which divisorial valuations are replaced by quasi-monomial valuations. Specifically, there always exists a quasi-monomial valuation computing the stability threshold by either \cite{BLX19} or \cite{BJ20,Xu20}.
\end{itemize}
In \cite{Xu20b}, the conjecture is reduced to the problem of showing that valuations computing the stability threshold induce finitely generated filtrations of the section ring.

\subsubsection{Relation with test configurations}

The first and third authors and Zhou previously studied the relationship between Conjecture 1.1 and  destabilizing test configurations.

\begin{thm}[\cite{BLZ19}]\label{t:BLZinf}
If $(X,D)$ is a log Fano pair that is K-unstable, then 
\begin{equation}\label{e:d-1=infF/min}
\delta(X,D) -1  = \inf_{(\cX,\cD)}  \frac{ \Fut(\cX,\cD)}{\mnorm{\cX,\cD}}
,\end{equation}
where the $\inf$ runs through non-trivial special test configurations of $(X,D)$. 
Additionally,
\begin{enumerate}
    \item the infimum in \eqref{e:d-1=infF/min} is computed if and only if Conjecture 1.1 holds for $(X,D)$;
    \item if $(\cX,\cD)$ computes the infimum, then $\delta(X,D) = \delta(\cX_0,\cD_0)$.
\end{enumerate}
\end{thm}

The results in \cite{BLZ19} are stated in terms of twisted K-stability \cite{Der16}, which we do not use in this paper. 
To translate results to our setting note that for $\beta \in (0,1]$,
\begin{itemize}
    \item[(i)] the \emph{$\beta$-twisted Futaki invariant} $\Fut_{1-\beta}(\cX,\cD)$ equals $\Fut(\cX,\cD) - (1-\beta) \mnorm{\cX,\cD}$  \cite[Proposition 3.6.1]{BLZ19}.
    \item[(ii)]  a log Fano pair is $\beta$-twisted K-semistable if and only if
$\Fut_{1-\beta}(\cX,\cD) \geq 0 $ for all special test configurations $(\cX,\cD)$ of $(X,D)$ \cite[Theorem 1.6]{BLZ19}.

\end{itemize}
\begin{proof}
The equality follows from \cite[Theorem 1.5]{BLZ19}. Statements (1) and (2) follow from \cite[Theorem 1.1]{BLZ19}.
\end{proof}


\begin{rem}\label{e:P11m}
Divisors computing the stability threshold, as well as special test configurations computing $\inf \frac{ \Fut}{\mnorm{\,}}$,
 do not in general give canonical ``destabilizations'' of  K-unstable log Fano pairs.
 For example, consider the $\Q$-Fano variety $X:=\PP(1,1,m)$ with $m>1$. 
Since $X$ is toric, one can check 
using \cite[Section 7]{BJ20}
that $X$ is K-unstable and  there are infinitely many divisors over $X$ computing its stability threshold. For instance, every ruling corresponding to a section of $\cO_X(1)$ is a divisor computing $\delta(X)=\frac{3}{m+2}$.
\end{rem}

\subsection{Stability  Function}\label{ss:bmu}

In this paper, we introduce a bi-valued function on the set of special test configurations of a log Fano pair.
Conjecturally, the function will  identify a unique optimal destabilization  of a K-unstable log Fano pair.


For a non-trivial special test configuration $(\cX,\cD)$ of $(X,D)$, we define the invariant
\[
\bm{\mu} \left( \cX,\cD \right) 
:=
\left( \mu_1 \left( \cX,\cD \right), \mu_2 \left( \cX,\cD \right) \right) \in \R^2,
\]
where 
\[
\mu_1 \left( \cX,\cD \right)
:=
 \frac{\Fut(\cX , \cD )}{ \mnorm {\cX,\cD} }
\quad \text{ and } \quad 
\mu_2 \left( \cX,\cD \right):=
\frac{\Fut(\cX , \cD )}{ \lnorm{\cX,\cD}}.
\]
For  a K-unstable log Fano pair  $(X,D)$, we set 
\begin{equation}\label{eq:M=infmu1}
M^{\bmu}(X,D) : = (M_1^{\bmu}(X,D),M_2^{\bmu}(X,D)) := \inf_{ (\cX,\cD)} \bmu(\cX,\cD)  \in (\mathbb{R} \cup \{\pm \infty\})^2
,
\end{equation}
where the infimum runs through non-trivial special test configurations of
$(X,D)$ and  is taken with respect to the lexicographic ordering on $\mathbb{R}^2$. When $(X,D)$ is K-semistable, we set $M^{\bmu}(X,D) : = {\bf 0} \in \mathbb{R}^2$.

Observe that if $(X,D)$ is K-unstable, then
\begin{equation}\label{e:M_1}
M_{1}^{\bmu}(X,D) = \inf_{(\cX,\cD)} \mu_1(\cX,\cD) = \delta(X,D)-1,
\end{equation}
where the second equality is Theorem \ref{t:BLZinf}.
Additionally, if the inf in  \eqref{e:M_1} is computed, then 
\[
M_{2}^{\bmu}(X,D) = \inf \{  \mu_2(\mathcal{X},\mathcal{D}) \, \vert \, 
\mu_1( \cX,\cD) = M_1^{\bmu} (X,D)\}
.\] If not, $M_{2}^{\bmu} (X,D)= +\infty$. Since $(X,D)$ is K-unstable and $\delta(X,D)>0$ \cite[Theorem A]{BJ20}, $M_{1}^{\bmu}(X,D) \in (-1,0)$. 
Using \eqref{e:mnorm<clnorm}, we see $M_{2}^{\bmu}(X,D) \in \R_{<0} \cup \{+\infty\}$.


\begin{rem}
In the literature, minimizers of $\mu_1$ and $\mu_2$ have been studied separately. 
\begin{enumerate}
    \item In \cite{BLZ19}, properties of minimizers of $\mu_1$
    are studied using tools from birational geometry.
    See  Theorem \ref{t:BLZinf}.
    \item In \cite{Don05,Sze08,Xia19},  optimal destabilizations refer to test configurations that minimize $\mu_2$ and are studied in the more general polarized case.
\end{enumerate}
While (1) has the advantage that  birational geometry results can be applied to
study minimizers of  $\mu_1$, the minimizer of $\mu_1$ is in general non-unique (see Remark
\ref{e:P11m}). This can be
fixed by leveraging the convexity properties of $\mu_2$ and using the above bi-valued function.
\end{rem}

\subsection{Moduli spaces of log Fano pairs}\label{ss:modulifano}

To define a moduli functor for log Fano pairs, we need to  define families of log Fano pairs over arbitrary schemes.
To have a well behaved moduli theory for the divisors that appear, we use \cite{Kol19}.

\begin{defn}\label{d-logfanofamily} We call  $f:(X,D:=c\Delta)\to T$ a \emph{family of log Fano pairs}
 if 
\begin{enumerate}
\item $f:X\to T$ is a flat projective morphism of schemes,
\item $\Delta$ is K-flat family of relative Mumford divisors on $X$ (see \cite{Kol19}),
\item $K_{X/T} + D$ is $\Q$-Cartier, and
\item $(X_{\overline{t}},D_{\overline{t}})$ is a log Fano pair for each $t\in T$.
\end{enumerate}
Furthermore, we call $f$ a \emph{family of K-semistable log Fano pairs}
if in addition
\begin{enumerate}
\item[(5)] $(X_{\overline{t}}, D_{\overline{t}})$ is K-semistable  for each $t\in T$.
\end{enumerate}
Above $c\in \Q_{>0}$.
Since $\Delta$ has  coefficients in $\N$, the coefficients of $D$ lie in $\{ n c \mid n  \in \N\} \cap [0,1]$.
\end{defn}

For $n\in \N$ and $V,c\in \Q_{>0}$, we define the moduli functor
$\cM_{n,V,c}^{\rm Fano}$ that sends $T\in {\sf Sch}_k $ to 
 \[
\mathcal{M}^{\rm Fano}_{n,V,c}(T) = 
 \left\{
  \begin{tabular}{c}
\mbox{families of log Fano pairs $(X, D:= c\Delta ) \to T$} \\
\mbox{with $\dim(X_t)=n$ and $(-K_{X_t}-D_t)^n = V$  for all $t\in T$}
\end{tabular}
\right\}
\]
and subfunctor 
$\cM_{n,V,c}^{\rm Kss}$ defined by
 \[
\mathcal{M}^{\rm Kss}_{n,V,c}(T) = 
 \left\{
 \begin{tabular}{c}
\mbox{families of K-semistable log Fano pairs $(X, D:= c\Delta ) \to T$} \\
\mbox{with $\dim(X_t)=n$ and $(-K_{X_t}-D_t)^n = V$  for all $t\in T$ }
\end{tabular}
\right\}
\]

The next theorem is a consequence of the following recent results:
the  boundedness of $\cM_{n,V,c}^{\rm Kss}$ \cite{Jia17,Che18,LLX18,XZ20}, the openness of K-semistability \cite{BLX19,Xu20}, 
 and the existence of a separated good moduli space \cite{BX19,ABHLX19}. See \cite[Theorem 2.21]{XZ19} for details. 

\begin{thm}\label{t:K-moduli}
The moduli functor $\cM_{n,V,c}^{\rm Kss}$ is an 
algebraic stack of finite type over $k$ with affine diagonal, and it admits a separated good moduli space $ M_{n,V,c}^{\rm Kps}$ whose $k$-valued points are in bijection with K-polystable log Fano pairs $(X,cD:= \Delta)$ of dimension $n$ and volume $V$.
\end{thm}

The next result on the invariance of certain Hilbert functions is a consequence of \cite{Kol17}.

\begin{prop}\label{p:invarianceEuler}
Let $(X,D:=c\Delta)\to T$ be a family of log Fano pairs. If $mD$ is a $\Z$-divisor,
then $t \mapsto  \chi(X_t, \cO_{X_t}(- m(K_{X_t}+D_t))$ is locally constant.
\end{prop} 

\begin{proof}
It suffices to consider the case when $T$ is the spectrum of a DVR. 
In this case,
 $\cO_{X}(- m(K_{X/T}+D))$ is flat over $T$ and $\cO_{X}\left(- m(K_{X/T}+D) \right)\vert_{X_t} \simeq \cO_{X_t}\left(- m(K_{X_t}+D_t)\right)$ for all $t\in T$ by  \cite[Proposition 2.76.2 and Definition 3.52]{Kol17}. Therefore, the function is constant on $T$.
\end{proof}
Let $c_{\rm den}$ denote the denominator of $c$. The above proposition implies that if  $[(X, D:= c\Delta ) \to T] \in \cM_{n,V,c}^{\rm Fano}(T)$, then
the Hilbert function 
\[
h: c_{\rm den} \Z\to \Z  \quad \text{ defined by } \quad h(m) : = \chi \left( X_t, \cO_{X_t}(-m(K_{X_t}+D_t)) \right)
\]
is independent of $t$ in a fixed connected component of $T$. For such a Hilbert function $h$, let $\mathcal{M}^{\rm Fano}_{h,c} \subset \mathcal{M}^{\rm Fano}_{n,V,c}$
and 
 $\mathcal{M}^{\rm Kss}_{h,c} \subset \mathcal{M}^{\rm Kss}_{n,V,c}$
 denote the subfunctors
 parametrizing families of log Fano pairs (resp., K-semistable log Fano pairs) with Hilbert function $h$. 
Note that \begin{equation}\label{eq:hilbert}
\cM_{n,V,c}^{\rm Fano}= \bigsqcup_{h} \cM_{h,c}^{\rm Fano}
\quad 
\text{ and } 
\quad 
\cM_{n,V,c}^{\rm Kss}= \bigsqcup_{h} \cM_{h,c}^{\rm Kss}
\end{equation}
 where the union runs through Hilbert functions $h$ of log Fano pairs $(X,D: = c \Delta)$ of dimension $n$ and volume $V$.

\begin{rem}
When $c=1$, $\cM_{n,V,1}^{\rm Kss}$ parametrize families of  K-semistable $\Q$-Fano varieties  of dimension $n$ and volume $V$, since if $(X,D)$ is a log Fano pair and $D$ is a $\Z$-divisors, then $D=0$. 
Note that
$\cM_{n,V,1}^{\rm Kss}$  differs from the moduli stack of K-semistable $\Q$-Fano varieties $\cM_{n,V}^{\rm Kss}$  considered in \cite{BX19,ABHLX19}. Indeed, while $\cM_{n,V,1}^{\rm Kss}$ parameterizes families satisfying the   so called Viehweg's condition, $\cM_{n,V}^{\rm Kss}$  considers families satisfying Koll\'ar's condition  (see \cite[\S1]{BX19}). These two conditions in general could be different (see \cite{AK19}), though they agree when the base is reduced by \cite[Theorem 3.68]{Kol17}. 
\end{rem}

 \subsection{Theta-stratificaitons}\label{ss-thetastratification}
The notion of a $\Theta$-stratification originated in \cite{HL14} and generalizes the Harder-Narasimhan stratification of the moduli of coherent sheaves on a projective scheme as well as the Kempf-Ness stratification in GIT to more general stacks.

\subsubsection{Filtrations}
Let   $\Theta$ denote the quotient stack $[\A^1/\G_m]$. 
Maps from $\Theta$ into a stack form the basis of stability notions in \cite{HL14}. In the case of K-stability, a special test configuration of a log Fano pair $[(X,D)] \in \cM_{n,V,c}^{\rm Fano}$ is equivalent to the data of a map $f: \Theta \to \cM_{n,V,c}^{\rm Fano}(k)$ with an isomorphism $f(1) \simeq [(X,D)]$.

For a stack $\sX$, let $\underline{\Map}(\Theta,\sX)$ denote the mapping stack parametrizing morphisms from $\Theta$ to $\sX$ and write ${\rm ev}_1$ and ${\rm ev}_0$ for the evaluation maps $\underline{\Map}(\Theta,\sX)  \to \sX$. We know that $\underline{\Map}(\Theta,\sX)$ is an algebraic stack when $\sX$ is a locally finite type algebraic stack over $k$ with affine stabilizers \cite{HLP20, HR19}. When $\sX = [X/G]$ is a quotient stack where $G$ is an algebraic group acting on a scheme $X$, the above mapping stack has a concrete description as follows (see  \cite[Theorem 1.37]{HL14}):
\[
\underline{\Map}(\Theta,[X/G])=\bigsqcup_{\lambda\in \Lambda} [X_\lambda/P_\lambda].
\]
Here $\Lambda$ is the complete set of conjugacy classes of 1-PS' $\lambda:\mathbb{G}_m\to G$, $X_\lambda$ is the union of Bialynicki-Birula strata of $X$ associated to $\lambda$ which equals $\{x\in X\mid \lim_{t\to 0}\lambda(t)\cdot x\textrm{ exists} \}$ set theoretically, and $P_\lambda=\{g\in G\mid \lim_{t\to 0} \lambda(t) g\lambda(t)^{-1}\textrm{ exists}\}$.

\subsubsection{Definition}
\begin{defn}\label{d:ThetaStrat}\footnote{This definition differs from that in \cite{HL14,AHLH18} by a sign convention to conform to the convention in the K-stability literature that non-negativity of the Futaki invariant corresponds to semistability.}
Let $\sX$ be an algebraic stack locally of finite type over $k$ with affine stabilizer groups.
\begin{enumerate}
    \item A \emph{$\Theta$-stratum} in $\sX$ is a union of connected components $\sS\subset \uMap(\Theta, \sX)$ such that ${\rm ev}_1:\sS\to \sX$ is a closed immersion. Informally, we sometimes identify $\sS$ with the closed substack ${\rm ev}_1(\sS) \subset \sX$.
    \item A \emph{$\Theta$-stratification} of $\sX$ indexed by a totally order set $\Gamma$ is a cover of $\sX$ by open substacks $\sX_{\geq c}$ for $c\in \Gamma$ such that $\sX_{\geq c'} \subset \sX_{\geq c}$ for $c'>c$, along with a $\Theta$-stratum $\sS_c\subset \uMap(\Theta, \sX_{\geq c})$ in each $\sX_{\geq c}$ whose complement in $\sX_{\geq c}$ is $ \cup_{c' >c} \sX_{\geq c'}$. We assume that for each $x\in |\sX|$ the subset $\{c \in \Gamma \, \vert \, x\in \sX_{\geq c}\}$ has a maximal element.    We assume for convenience that $\Gamma$ has a maximal element
$0\in \Gamma$.
    \item A $\Theta$-stratification is \emph{well-ordered} if for each $x\in |\sX|$, every nonempty subset of $\{ c\in \Gamma \, \vert \, {\rm ev_1}(\sS_{c}) \cap \overline{ \{x \} }\neq \emptyset \}$ has a maximal element.
\end{enumerate}
\end{defn}
Given a $\Theta$-stratification, we denote by $\sX^{\rm ss}: = \sX_{\geq 0}$ the \emph{semistable} locus of $\sX$. 
For any $x\in \sX(k) \setminus \sX^{\rm ss}(k)$, the unique stratum $\sS_c$ such that $x\in {\rm ev}_1(\sS_c)$ determines a canonical map $f:\Theta  \to \sX$ with $f(1)=x$. This map is referred to as the \emph{HN-filtration} of $x$ in \cite{HL14}.

\begin{example}
Let $C$ be a smooth projective curve over a field $k$, and let $\sX$ be the stack of vector bundles on $C$ of rank $n$ and degree $d$. Then the Harder-Narasimhan-Shatz stratification of $\sX$ \cite{HN74,Sha77} is a $\Theta$-stratification. For an unstable vector bundle $E$ on $C$, the Harder-Narasimhan filtration is the unique filtration $0\subsetneq E_p \subsetneq \cdots \subsetneq E_0 = E$ such that $i^{th}$ associated graded piece $E_i/E_{i+1}$ is semistable and locally free, and the slope $\mu_i := \deg(E_i/E_{i+1})/ \mathrm{rank}(E_i/E_{i+1})$ is strictly increasing with $i$. This encodes a map $f : \Theta_k \to \sX$ as follows:

One considers the graded sheaf of $\cO_{C}[t]$-modules $\cO_{C}[t^{\pm 1}] \otimes_{\cO_{C}} E$. Each $E_i$ defines a graded submodule $\cO_{C}[t] \otimes_{\cO_{C}} E_i$, and we combine these into the graded submodule
\[
\mathcal{E} := \sum_{i=0}^p t^{-n! \mu_i} \cdot \cO_{C}[t] \otimes_{\cO_{C}} E_i \subset E \otimes_{\cO_{C}} \cO_{C}[t^{\pm 1}].
\]
The factor of $n!$ guarantees that all of the exponents are integers. One can check that the $\mathbb{G}_m$-equivariant quasi-coherent sheaf on $\mathrm{Spec}_{C}(\cO_{C}[t]) \cong \mathbb{A}^1 \times C$ corresponding to $\mathcal{E}$ is locally free. Hence this equivariant sheaf defines a map $f : \Theta_k \to \sX$.

It is shown in \cite[Section 6]{HL14} that the set of maps constructed in this way defines an open substack $\sS \subset \uMap(\Theta,\sX)$, and $\mathrm{ev}_1 : \uMap(\Theta,\sX) \to \sX$ identifies each connected component of $\sS$ with the corresponding Harder-Narsimhan-Shatz stratum in $\sX$.
\end{example}

\section{Log Fano pairs with torus actions}\label{s:logFanotorus}
In this section, we collect basic results on the behaviour of the Futaki invariant, minimum norm, and $L^2$ norm for one paramater subgroups of a torus acting on a log Fano pair. 
While the results in Sections \ref{ss:momentp} and \ref{ss:quadform} are well known in the K-stability and K\"ahler-Einstein metrics literature (see e.g. \cite{FM95, WZ04, Sze08}), we provide short algebraic proofs for the convenience of the reader.

Let $(X,D)$ be an $n$-dimensional log Fano pair with an action of a $d$-dimensional torus $\T:= \G_m^d$. 
We write $N: = \Hom(\G_m, \T)$ and $M:= \Hom(\T,\G_m)$ for the \emph{coweight} and \emph{weight lattices}.
The lattices are isomorphic to $\Z^d$ and admit a perfect pairing ${\langle \,  ,  \rangle : M \times N \to \Z}$. 
 For $\mathbb{K} \in \{ \Q,\R \}$,  write
 $N_{\mathbb{K}}:= N \otimes_\mathbb{Z} \K$ 
 and  
 $M_{\K}:= M \otimes_\mathbb{Z} \K$
  for the  corresponding  vector spaces.
  
 Fix a positive integer $r$ such that $L:=-r(K_X+D)$ is a Cartier divisor and write
\[
R(X,L)
:= 
\bigoplus_{m \in \N} R_m 
=
\bigoplus_{m \in \N} H^0\left(X,\cO_X(L)\right)
\]
for the section ring of $L$. 
Set $N_m:= \dim R_m$ for each $m\geq1$.

The $\T$-action on $X$ 
induces a canonical action on each vector space $R_m$. This gives a direct sum decomposition
$R_m = \bigoplus_{u\in M} R_{m,u}$, where 
\[
R_{m,u}
:=
\{ s \in R_m \, \vert \, {\bf t} \cdot s = u({\bf t}) s \text{ for all } {\bf t}\in \T \}
\]
is the $u$\emph{-weight space}, satisfying $R_{m,u}\cdot R_{m',u'} \subseteq R_{m+m',u+u'}$. 

Note that an element  $v\in N := \Hom(\G_m, \T)$ induces an action of $\G_m$ on $(X,D)$ and hence $R$. 
If we write $R_m = \bigoplus_{\lambda \in \Z} R_{m,\lambda}$ for the weight decomposition with respect to the $\G_m$-action induced by $v$, then 
\begin{equation} \label{e:weights}
    R_{m, \lambda }= \bigoplus_{u \in M \vert \langle u,v\rangle = \lambda} R_{m,u}.
\end{equation}

\subsection{Moment polytope and  barycenter}\label{ss:momentp}
For each integer $m \geq 1$, we set 
\[
P_m
:= 
{\rm conv.hull}( u \in M \, \vert \, R_{m,u}\neq 0 ) \subseteq M_{\R}
.\]
The \emph{moment polytope}  of $(X,D)$ with respect to $\T$ is given by
\[
\textstyle
P: = {\rm conv.hull}\Big( \bigcup_{m\geq 1 }   \frac{1}{mr} \cdot P_m \Big)
. \]
Since $R$ is a finitely generated algebra,
$P$  may be expressed as the convex hull of finitely many points in $M_\Q$.
Furthermore, $P= \frac{1}{mr}P_m$ for $m\geq1$ sufficiently divisible. 
The \emph{weighted barycenter} of $P$ is given by 
\begin{equation}\label{eq:baryquad}
b_P
:= 
\lim_{m \to \infty}    \frac{1}{ mr N_m} \sum_{u \in M}   \dim(R_{m,u}) u,
\end{equation}
where the limit is taken in $M_\R$. 

\begin{lem}\label{l:baryclim}
The above limit  exists and lies in $M_{\Q}$. Additionally, for $v\in N$, 
\[
\Fut(X,D;v) 
= 
- \langle b_P, v \rangle 
\quad \text{ and } \quad
\mnorm{X,D;v}
 = 
 \langle b_p, v\rangle - \min_{u \in P} \langle u,v\rangle.
\]
\end{lem}

\begin{proof}
To see that the limit exists and lies in $M_{\Q}$, it suffices to show that, for each $v\in N$,
\begin{equation}\label{p:limitsv}
\lim_{m \to \infty} \frac{1}{mr N_m}  \sum_{u \in M} \dim (R_{m,u}) \langle u,v\rangle 
\end{equation}
is a rational number.
To prove the latter, fix $v\in N$ and  write $R_m = \bigoplus_{\lambda \in \Z} R_{m,\lambda}$ for the weight decomposition with respect to the $\G_m$-action induced by $v$.
Using \eqref{e:weights}, we see
\[
\lim_{m \to \infty}   \frac{1}{mr N_m} \sum_{u \in M} \dim(R_{m,u})  \langle u,v\rangle 
=
 \lim_{m \to \infty} \frac{1}{mr N_m}  \sum_{\lambda \in \Z}   \lambda \dim (R_{m,\lambda})
 = -\Fut(X,D;v),
\]
which is  rational. Therefore,  $b_P \in N_\Q$ and the formula for the Futaki invariant holds.

To deduce the formula for the minimum norm, note that  
$\lambda_{\min,m} = 
\min_{u \in P_m}  \langle u,v \rangle.$
Since $P= \frac{1}{mr} P_m$ for $m\geq 1$ sufficiently divisible, $\lambda_{\min} = \min_{u \in P}  \langle u,v \rangle$ and the formula follows. \end{proof}
 
 \subsection{Associated quadratic form}\label{ss:quadform}
The \emph{associated quadratic form} $Q: N_{\R} \to \R$ of the weight decomposition is defined by 
\[
Q(v) 
: =
 \lim_{m \to \infty} \frac{1}{ (mr)^2 N_m}   \sum_{u \in M}  \dim( R_{m,u})  \langle u-mr b_P, v \rangle ^2
\] 
 
 \begin{lem}\label{l:Qlim}
The function $Q$ is a rational non-negative quadratic form (in particular, the above limit exists). 
Furthermore, $ Q(v) =\lnorm{X,D;v}^2$ for any $v\in N$.
 \end{lem}

\begin{proof}
For each $m\geq 1$, consider the function $Q_m:N_{\R}\to \R$ defined by 
\[
Q_m(v) 
= \frac{1}{ (mr)^2N_m}   \sum_{u \in M} \dim (R_{m,u})\langle u-mr b_P, v \rangle ^2 
\]
is a non-negative rational quadratic form. 
To see that $(Q_m)_m$ converges to a non-negative quadratic form,
 it suffices to show that $\lim_m Q_m(v) \in \Q$ for each $v\in N$.

To prove the latter, fix $v \in N$ and write 
 $R_m = \bigoplus_{\lambda \in \Z} R_{m,\lambda}$ for the weight decomposition with respect to the $\G_m$-action induced by $v$.
 Note that 
 \begin{align*}
 Q(v) &= 
 \lim_{m \to \infty}
 \frac{1}{(mr)^2 N_m} \sum_{u\in M} 
 \left( \langle u, v\rangle^2
-2mr  \langle u,v \rangle \cdot \langle b_P, v \rangle  + (mr)^2\langle  b_P, v \rangle^2 \right) \dim R_{m,u} \\
 & = 
 \lim_{m \to \infty} 
 \left( \frac{ 1}{(mr)^2 N_m} \sum_{u \in M} \langle u,v \rangle^2 \dim R_{m,u}
 \right) - \langle b_P,v\rangle^2
 \end{align*}
Using  \eqref{e:weights},  the right hand side above  equals
\[
\lim_{m \to \infty} \frac{1}{ (mr)^2 N_m}   \sum_{\lambda \in \Z}   \lambda^2 \dim R_{m,\lambda}
-\left(
\lim_{m \to \infty} \frac{1}{ mr N_m}   \sum_{\lambda \in \Z}  
 \lambda \dim R_{m,\lambda} \right)^2
\]
 which is precisely  $\lnorm{X,D;v}^2$. Since
 the latter value is rational, the result follows.
\end{proof}

 \subsection{Stability function} \label{S:stability_function}
We consider the three functions  $N_{\R} \to \R$ given by
\[
\Fut(v): =- 
\langle b_P, v \rangle,
\quad   
\mnorm{v}:=
\langle b_P ,v \rangle - \min_{u \in P} \langle u,v \rangle , 
\quad 
\text{ and }\quad
\lnorm{v} = \sqrt{Q(v)}
.\]
By Lemmas \ref{l:baryclim} and \ref{l:Qlim},
these agree with the corresponding invariants defined in Section \ref{ss:furtherinvariants}, when $v\in N$.
Observe that (i) $\Fut(\, \cdot \,)$ is rational linear, 
(ii) $\mnorm{ \, \cdot \,}$ is rational piecewise linear (by \emph{rational piecewise linear}, we mean that there is a decomposition of $N_{\R}$ into rational polyhedral cones, such that on each cone the function is rational linear)
and convex, 
and (iii) $\lnorm{ \, \cdot \, }^2$ is a non-negative rational quadratic form. (Note that the convexity in (ii) follows from the fact that $v\mapsto \langle b_p,v\rangle$ is linear and $v\mapsto \min_{u\in P} \langle u,v\rangle$ is concave, since it is the minimum of a collection of concave functions.)

\begin{lem}\label{l:positive}
If the natural map $\T \to \Aut(X,D)$ has finite kernel, then $\lVert \,\cdot \, \rVert_{\mathrm{m}}$ and  $\lVert \, \cdot \, \rVert_{2}$  are  positive on $N_\R\setminus 0 $.
\end{lem}

\begin{proof}
Since $\T\to \Aut(X,D)$  has finite kernel,
each $v\in N\setminus 0$  induces a non-trivial $\G_m$-action on $(X,D)$.
Therefore,  $\mnorm{v}$ and $\lnorm{v}$ are strictly positive on $N\setminus 0$.
Using properties (ii) and (iii) above, it follows that the functionals are also strictly positive on $N_\R\setminus 0$.
 \end{proof}

When $\T\to \Aut(X,D)$ has finite kernel, we consider the bi-valued \emph{stability function} $\bm{\mu}: N_{\R} \setminus 0 \to \R^2$ defined by 
$\bmu(v) : =  \left( \mu_1(v), \mu_2(v) \right)$
\[
\mu_1 (v) := \frac{ \Fut(v)}{\mnorm{v}}
\quad \quad \text{ and } \quad \quad
\mu_2 (v) := \frac{ \Fut(v)}{\lnorm{v}}.
\]
and  endow $\R^2$ with the lexicographic order. 
Since $\mu_1$ and $\mu_2$ are invariant with respect to scaling by $\mathbb{R}_{>0}$, $\bmu$ induces a function on $\Delta_\mathbb{R}: = (N_{\mathbb{R}}\setminus 0)/ \mathbb{R}_{>0}$.

We note the following quasi-convexity property of $\bmu$.

\begin{prop}\label{p:quasiconvex}
Assume $\T \to \Aut(X,D)$ has finite kernel.
Fix points $v,w  \in N_{\R}\setminus 0$ with distinct images in $\Delta_{\R}$ and $t\in (0,1)$.  If $\Fut(v)$ and $\Fut(w)$ are $<0$, then 
\[
\mu_{i} (tv+(1-t)w) \leq \max\{ \mu_i( v) , \mu_i(w) \} \quad \text{ for } i=1,2 .
\]
Furthermore, if $i=2$, then the inequality is strict.
\end{prop}

\begin{proof}
After scaling $v$ and $w$ by $\R_{>0}$ we may assume $\Fut(v)=\Fut(w)$ and equals $\Fut( t v+(1-t) w)$ by linearity. 
Next, note that  $\mnorm{\,\cdot \,}$  is convex and  $\lnorm{\, \cdot \,}^2$ strictly convex (since it a quadratic form and positive definite by Lemma \ref{l:positive}).
Therefore, for the two norms satisfy
\[ \lVert tv+(1-t)w \rVert \leq \max\{ \lVert v \rVert , \lVert w \rVert \}\] 
and the inequality is strict for the $L^2$ norm. This implies the desired inequalities.
\end{proof}

For a  cone $\sigma\subset N_{\R}$, we set  $\Delta(\sigma):= (\sigma \setminus 0)/ \R_{>0} 
\subset \Delta_\R$, and refer to $\Delta(\sigma)$ as the image  of $\sigma$ in $\Delta$. We proceed to describe the geometry of minimizers of $\bmu$ restricted to such subsets. 

\begin{prop}\label{p:minimizeroncone}
Assume $\T \to \Aut(X,D)$ has finite kernel. 
If $\sigma \subseteq N_\R$ is a rational polyhedral cone with $ \sigma  \cap \{ \Fut <0\}\neq \emptyset$, then  the infimum
\begin{equation}\label{eq:infbmu}
 \inf_{v \in \Delta(\sigma) } \bmu(v).
\end{equation}
is achieved at a unique point in $ \Delta(\sigma)$ and the point is rational.
\end{prop}

Before proving the proposition, we describe the geometry of minimizers of $\mu_1$ and $\mu_2$ separately. Note that the rationality of $\Delta_2$ is known; see e.g. \cite{FM95}.

\begin{lem}\label{l:minimizercone}
Keep the assumptions of Proposition \ref{p:minimizeroncone} and 
set 
\[
\Delta_i : = \big\{ v\in \Delta_{\sigma} \, \vert \, 
\mu_i(v) = \inf_{v \in \Delta(\sigma)} \mu_i(v) \big\}
\quad \text{ for } i=1,2.\] 
Then $\Delta_1$ is the image of a nonempty rational polyhedral cone and  $\Delta_2$ is a  rational point.
\end{lem}

\begin{proof}
Since $\Fut$ is rational linear and $\mnorm{\, \cdot \, }$ is piecewise rational linear, 
the value  $ M_1 := \inf \{ \mu_1(v) \, \vert \, 
 v\in \Delta(\sigma) \}
$ is rational. Additionally, $M_1<0$, since $\sigma \cap \{ \Fut<0 \}\neq \emptyset$.

 Now, note that the function $g:N_\R \to \R$ defined by
\[
g(v):= \Fut(v) - M_1 \mnorm{v}
.\]
 is non-negative on $\sigma$ and $\sigma_1: =\{ v\in \sigma \, \vert \, g(v)=0\}$ has image $\Delta_1$ in $\Delta_\R$.  Since $g$ is rational piecewise linear and convex, it follows that $\sigma_1$ is a rational polyhedral cone, which completes the  $i=1$ case.
 
Next, note that $\Delta_2$ is nonempty, since $\mu_2$ is a continuous function and $\Delta(\sigma)$ is compact. 
Furthermore, $\Delta_2$ must be a point, by Lemma \ref{p:quasiconvex}. The rationality of the point  follows from the fact that $\lnorm{\,\, }^2$ is a  rational quadratic form  and a Lagrange multiplier argument (see the proof of \cite[Lemma 4.12]{HL14}).
\end{proof}

\begin{proof}[Proof of Proposition \ref{p:minimizeroncone}]
By Lemma \ref{l:minimizercone},   $  \inf \{ \mu_1(v) \, \vert \, v\in \Delta(\sigma) \}$ is achieved on a set $\Delta_1 \subset \Delta_\R$, which is the image of a   nonempty rational polyhedral cone. 
Since $\R^2$ is endowed with the lexicographic order,  $v\in \Delta(\sigma)$ achieves
$\inf\{ \bmu(v)\, \vert \,v\in \Delta(\sigma) \}$ if and only if $v\in \Delta_1$ and $v$ achieves $\inf  \{  \mu_2(v) \,\vert \, v \in \Delta_1 \}$.
 Applying Lemma \ref{l:minimizercone} again gives that  \eqref{eq:infbmu} is achieved at a unique point  and the point is rational.
\end{proof}

\section{Existence of minimizers and constructibility results}\label{s:existance}
In this section we prove Theorem \ref{t:HNfilts}.1 on the existence of test configurations computing $M^{\bmu}$. In the process, we also prove a constructibility result for $M^{\bmu}$.

\subsection{Parameter space}\label{ss:paramtest}
Fix $c \in \Q_{>0}$ and a Hilbert function $h:c_{\rm den}\Z\to \Z$, where $c_{\rm den}$ denotes the denominator of $c$. 
For $0<\epsilon \leq 1$, let 
${ \cM^{\delta \geq\epsilon }_{h,c} \subseteq \cM^{\rm Fano}_{h,c}}$ denote the subfunctor
defined by 
\[
 \cM^{\delta \geq \epsilon }_{h,c}(T) = \{ [(X,D:= c\Delta) \to T] \in \  \cM^{\rm Fano}_{h,c}(T) \, \vert \, \delta(X_{\overline{t}},D_{\overline{t}}) \geq \epsilon \, \text{ for all }\, t\in T \}.\]
When $\epsilon=1$, this is precisely $ \cM^{\rm Kss}_{h,c}$. 
In order to parameterize  test configurations of log Fano pairs, we will explicitly describe $ \cM^{\delta \geq \epsilon }_{h,c}$ as a quotient stack. 

 To begin, we note the following boundedness result,
 which is is known (see \cite{Jia17, Che18,LLX18,XZ20}). 
 For the reader's convenience, we include a proof here.
\begin{thm}\label{t:bounded}
The moduli functor  $ \cM^{\delta\geq\epsilon }_{h,c}$ is bounded.
\end{thm}

\begin{proof}
By \cite{XZ20}*{Theorem 1.5}, there exists $N := N(h,\epsilon)\in \mathbb N$ such that for any $[(X,D:= c\Delta)] \in \cM_{h,c}^{\rm Fano}(k)$ with $\delta(X,D)\geq  \epsilon$, $N\cdot c_{\rm den}(-K_X-D)$ is an ample Cartier divisor. Then by \cite{HMX14}*{Corollary 1.8}, we know that the set of such pairs $(X,\Delta)$ is bounded. 
\end{proof}

As a consequence of Theorem \ref{t:bounded}, there exists an integer $r:= r(h,c,\epsilon )$ 
such that if $[(X,D:=c\Delta)] \in  |\cM_{h,c}^{\delta\geq\epsilon}|$,
then $L:=-r(K_X+ D)$ is a very ample Cartier divisor.
In addition, the set of degrees  $d:= \Delta\cdot L^{\dim X-1}$, where $[(X,D:=c\Delta)] \in | \cM_{h,c}^{\delta\geq \epsilon}|$ is finite.

To construct $\cM_{h,c}^{\delta\geq \epsilon}$ as a quotient stack, 
observe that if 
$[f:(X,D:=c\Delta) \to T ] \in \cM_{h,c}^{\delta\geq \epsilon}(T)$
and $ L:=-r(K_{X/T}+D)$, then ${H^i(X_t, \cO_{X_t}(L_t))=0}$ for all $i>0$ and $t\in T$ by Kawamata-Viehweg vanishing. 
Therefore, $f_*\cO_{X}(L)$ is a vector bundle of rank $h(r)$ and there is an embedding 
$X \hookrightarrow \PP( f_*\cO_{X}(L))$. 

To parametrize such objects, let $\bH:= \Hilb_{h(r \, \cdot \,)}(\PP^m)$ denote the Hilbert scheme parametrizing subschemes $X\subset \PP^{m}$ with Hilbert polynomial $h(r \, \cdot \, )$, where $m:= h(r)-1$. 
Write  $\bH^\circ \subseteq \bH$ for the open locus  parametrizing normal varieties so that $H^i(X,\cO_X(1))=0$ for all $i>0$ and $X\subset \PP^m$ is linearly normal.  Write $X_{\bH^\circ } \subset \PP^m_{\bH^{\circ}}$ for the corresponding universal family.
By \cite[Theorem 98]{Kol19}, there is a separated finite type $\bH^\circ$-scheme $\bP$ parametrizing  K-flat relative Mumford divisors on $X_{\bH^\circ }/\bH^\circ $ of all possible degrees $d$ as above. Let  $\Delta_{\bP}$ denote the corresponding universal K-flat relative Mumford divisor on the pullback $X_{\bP} \subset \PP^{m}_{\bP}$.

There exists a locally closed subscheme $Z\hookrightarrow \bP$ such that a morphism $S\to \bP$ factors through $Z$ if and only if
\begin{itemize}
    \item[(a)] $[ f_{S}: (X_{S} , D_{S}= c\Delta_S)\to S] \in 
\cM_{h,c}^{\delta \geq \epsilon }(S)$ and
\item[(b)] $\cO_{X}(-r(K_{X/S}+D_S)) \otimes f_S^*\mathcal{N}\simeq \cO_{X}(1)$ for some line bundle $\mathcal{N}$ on $S$.
\end{itemize}
Indeed, this follows from  \cite[Proof of Theorem 2.21]{XZ19}, with the boundedness given by Theorem \ref{t:bounded} and the openness by Theorem 1.1 of  \cite{BLX19} (rather than Corollary 1.2 of \emph{loc. cit} in \cite{XZ19}). We note that \cite[Section 6]{BLZ19} has a similar construction of $Z$ which works over (semi-)normal base schemes. Here, similar to \cite{XZ19}, we use K-flatness introduced in \cite{Kol19} to work over arbitrary base schemes.
Now, observe that the $S$-valued points of  $Z$  parameterize pairs
\[
\left( [ f: (X , D= c\Delta)\to S] \in \cM_{h,c}^{\delta \geq \epsilon }(S)\,\,;\,\, \phi \in {\rm Isom}_S( \PP( f_* \cO_X(L)) \simeq \PP_S^m) \right)\]
and that the $\PGL:=\PGL_{m+1}$-action on $\PP^{m}$
induces an action on $\bP$ that restricts to $Z$. 
 It follows from the construction that $\cM_{h,c}^{\delta\geq \epsilon } \simeq [ Z/ \PGL]$.
  Furthermore, by the construction of $\bH$ and $\bP$ (see \cite[Setion 6]{Kol19} for the latter),
 there exists an equivariant embedding $\bP \hookrightarrow \PP(W)$, where $W$ is a vector space with a $\PGL$-action.

\subsection{1-parameter subgroups of $\PGL$}\label{ss:1psPGL}
For a 1-PS $\lambda: \G_m \to \PGL$ and a closed point $z = [(X,D)\hookrightarrow \PP^{m}]\in Z$,
consider the $\G_m$-equivariant map 
\[
\A^1 \setminus 0\to Z\hookrightarrow  \PP(W)
\quad \text{ defined by }  \quad t \cdot z\mapsto \lambda(t) \cdot z.\]
By the valuative criterion for properness, there is a unique extension to a $\G_m$-equivariant morphism $\A^1 \to \PP(W)$.
In particular  $z_0: = \lim_{t\to 0} \lambda(t)\cdot z$ exists and  is fixed by $\la$.  

 If $z_0\in Z$,
 then $z_0$ corresponds to a  log Fano pair 
 $[(X_0,D_0)\hookrightarrow \PP^m]$ with a $\G_m$-action
 and the pullback of $(X_{\bP},c\Delta_{\bP})$ by $\A^1 \to Z$ is naturally a special test configuration of $(X,D)$ that we denote by $(\cX_{\la},\cD_{\la})$. 
 In this case, we set 
 \[
 \bmu(z,\lambda ) := \bmu (X_0, D_0; \la) = \bmu (\cX_{\la} , \cD_\la) \in \R^2 
 .\]
If $z_0\notin Z$, we set $\bmu(z,\lambda):=(+\infty, +\infty)$.

 \begin{prop}\label{p:M(X,D)=infl}
For $z=[(X,D) \hookrightarrow \PP^{m} ] \in Z$,
\[
M^{\bmu}(X,D)\leq 
\inf_{ \la \in \Hom( \G_m , \PGL)}  \bmu(z,\la)
.\]
Furthermore, if $(X,D)$ is K-unstable and satisfies  Conjecture \ref{c:optdest}, then the equality holds.
\end{prop}

\begin{proof}
The  first statement is clear, since  $\bmu(z,\lambda) = \bmu(\cX_\lambda,\cD_\lambda)$
for any 1-PS $\lambda$ such that $\lim_{t\to 0} \lambda(t) \cdot z\in Z$.
If $(X,D)$ is K-unstable and satisfies Conjecture \ref{c:optdest},
then Theorem \ref{t:BLZinf} implies $\inf_{(\cX,\cD)}\mu_1(\cX,\cD)$ is achieved.
Therefore,
\[
\inf_{(\cX,\cD)} \bmu (\cX,\cD) =  \inf_{ \mu_1 (\cX,\cD) = M_1^{\bmu}(X,D)} \bmu (\cX,\cD) .
\]
Now, fix a test configuration  $(\cX,\cD)$ with $\mu_1(\cX,\cD) = M_1^{\bmu}(X,D)$. By Theorem \ref{t:BLZinf}.2,
\[
\delta(\cX_0,\cD_0) = \delta(\cX,\cD)\geq \epsilon 
.\]
Therefore, $\cL_0: = -r(K_{\cX_0} + \cD_0)$ is a very ample  Cartier divisor and so is ${-r(K_{\cX/\A^1}+ \cD)}$.
The latter implies $(\cX,\cD) \simeq (\cX_{\lambda},\cD_{\lambda})$ for some 1-PS $\lambda:\G_m \to \PGL_n$  (see e.g. \cite[Section 3.2]{BHJ17}) and, hence, $\bmu(\cX,\cD) = \bmu(z,\lambda)$. This shows the inequality ``$\geq$'' also  holds.
\end{proof}

Using a standard argument from Geometric Invariant Theory, we will prove properties of the infimum appearing in Proposition \ref{p:M(X,D)=infl}.

\subsection{1-parameter subgroups of a maximal torus}\label{ss:1psT}
Fix a maximal torus $\T \subset \PGL$. Set $N:= \Hom(\G_m,\T)$ and
 $M: = \Hom(\T,\G_m)$.
Observe that
 \begin{equation}\label{e:infmulav}
 \inf_{ \la\in \Hom(\G_m, \PGL)} \bmu(z,\la) = 
 \inf_{g\in \PGL} \inf_{v \in N} \bmu(gz,v).
 \end{equation}
Indeed, this is a consequence of the following facts:
\begin{itemize}
    \item[(i)] $\bmu(z,\lambda ) = \bmu(gz, g\lambda g^{-1})$
for any $g\in \PGL$ and  $\lambda \in \Hom(\G_m,\PGL)$;
    \item[(ii)] for any $\lambda\in \Hom(\G_m, \PGL)$, there exists $g\in \PGL$ such that $g \lambda g^{-1}\in N$.
\end{itemize}
By \eqref{e:infmulav}, to study $\inf_{\la} \bmu(z,\la)$ it suffices to study $\inf_{v} \bmu(z,v)$ as we vary $z\in Z$.

\begin{prop}\label{p:mubehaviour}
Keep the above notation.
\begin{itemize}
    \item[(1)]  There exists a decomposition of $Z=\bigsqcup_{i=1}^s Z_i$ into locally closed subsets such that
    \[
     Z_i  \times N  \ni (z,v) \mapsto \bmu(z,v)
    \quad \text{ is independent of } \quad z\in Z_i.
    \]
    \item[(2)] Fix a point $z\in Z$. 
If the function $\bmu: N  \to \R^2$ defined by $v\mapsto \bmu(z,v)$ takes a value $<{\bf 0}$, then it achieves a minimum. 
\end{itemize}
\end{prop}

 Related results appear in the literature for when $\bmu$ is replaced by  the Futaki invariant.  For example, see  \cite[Lemma 2.10]{Oda13}.

Before proving the above result,
we recall an explicit  description  of the relevant 1-parameter degenerations \cite[pg. 51]{GIT}.
Since $\T$ acts linearly on $W$, we may choose a basis $\{e_1,\ldots, e_l\}$ for $W$
and cocharacters $u_1,\ldots,u_l \in M$ such that 
\[
{\bf t}\cdot e_i = u_i({\bf t} ) e_i \quad \quad \text{ for each } 1\leq i \leq l \text{  and } {\bf t }\in \T.\]
Hence, if we write a point  $[w]=[w_1:\cdots : w_l]\in \PP(W)$ using coordinates in this basis and fix $v\in N$, then
\[
 v(t) \cdot [w] = [t^{ \langle u_1, v \rangle } w_1 : \cdots : t^{ \langle u_l, v\rangle } w_l]
 \quad \text{ for } t\in \G_m
 .
\]
Therefore, if we set $I : = \{1 \leq  i \leq  l \, \vert \, w_i \neq 0 \}$,
then 
$
\lim_{t\to 0} v(t) \cdot [w] =  [w']$,
where 
\[
w'_{j} = \begin{cases}
 w_j & \text{ if } \langle u_j, v \rangle
\leq \langle u_i,v\rangle \text{   for all  } i \in I\\
0 & \text{otherwise}
\end{cases}
,\]
 and $v$ fixes $[w]$ if and only if $\langle u_i,v\rangle = \langle u_j,v\rangle$ for all $i,j\in I$.
 Based on this computation, 
for each  nonempty $I \subset \{1,\ldots, l \}$, we set
    \[U_I :=\{ [w] \in \PP(W)\, \vert \,   w_i \neq 0 \, \text{ iff }  i \in  I\}
     .\]
and, when $J \subset I$,  write 
    $\varphi_{I,J}:  U_{I} \to  U_J$ for the projection map. 

\begin{proof}[Proof of Proposition \ref{p:mubehaviour}.1]
    For each  nonempty subset $I \subset \{1,\ldots, l \}$, consider the locally closed subset  $Z_I := U_I \cap Z \subset Z$. 
    Next, write
$Z_I =\sqcup_{k} Z_{I,k}$ as the disjoint union of finitely many connected locally closed subschemes
such that, for each $J \subsetneq I$, 
 $\varphi_{I,J}(Z_{I,k})$  is either contained entirely in $Z$ or in $\PP(W)\setminus Z$.

To see the decomposition $Z=\sqcup_{I,k} Z_{I,k}$ satisfies the conclusion of the proposition, fix a component $Z_{I,k}$ and $v\in N$.
Set \[
J := \{ j \in I\, \vert \, \langle v,u_{j} \rangle \leq \langle v,u_i  \rangle \text{ for all } i \in I \}\subset I\]
and note that (i) if $z\in Z_{I,k}$, then $\lim_{t\to 0} v(t) \cdot z = \varphi_{I,J}(z)$ and (ii) $v$ fixes the points in $Z_J$.
If  $\varphi_{I,J}(Z_{I,k}) \subset Z$, then $\varphi_{I,J}(Z_{I,k})$
lies in a connected component of $Z_J$, since $Z_{I,k}$ is connected. 
In this case, $\bmu(z,v) = \bmu( \varphi_{I,J}(z), v)$ and the latter is independent of $z\in Z_{I,k}$ 
by Lemma \ref{l:futnormfiberwise}.
On the other hand, if $\varphi_{I,J}(Z_{I,k})\subset \PP(W)\setminus Z$,  then $\bmu(z,v) = (+\infty,+\infty)$ for all $z\in Z_{I,k}$. 
Therefore, the decomposition is of the desired form.
\end{proof}

Before proving Proposition  \ref{p:mubehaviour}.2, 
we  recall how  $\lim_{t \to 0} v(t)\cdot z$ changes as we vary $v\in N$.
Fix a point $[w]\in \PP(W)$ and consider the polytope
\[ Q : = {\rm conv.hull}(u_i \, \vert \, w_i \neq 0 ) \subseteq M_{\R}.\]
For a face $F\subset Q$, the normal cone to $F$ is given by 
\[
\sigma_F : = \{ v\in N_{\R} \, \vert \, \langle u,v\rangle \leq \langle u',v \rangle  \text{ for all } u\in F \text{ and } u' \in Q  \}
\]
and is a rational polyhedral cone.
Note that the cones $\sigma_F$ as $F$ varies through faces of $Q$ form a fan supported on $N_{\R}$. 
For a face $F\subset Q$, set 
\[
w^F_{j} = \begin{cases}
 w_j & \text{ if } u_j \in F\\
0 & \text{otherwise}
\end{cases}
.\]
Note that, if $v\in {\rm Int}( \sigma_F) \cap N$, then $\lim_{t\to 0} v(t) \cdot [w]= [w^F]$.
Additionally, if $v\in  { {\rm span}_{\R}(\sigma_F) \cap N}$, then $v$ fixes $[w^F]$.

\begin{proof}[Proof of Proposition \ref{p:mubehaviour}.2]
Let $[w]\in \PP(W)$ be a representation of $z$ in coordinates
and consider the polytope $Q\subset N_{\R}$ as above. 
Now, fix a face $F\subset Q$.
It suffices to show that if $\bmu$ takes a value $<{\bf 0}$ on ${\rm Int} (\sigma_{F})\cap N$, then $\inf \{ \bmu([w],v) \, \vert \,  v \in \sigma_F \cap N\}$ is a minimum. 

We claim that
\begin{equation}\label{eq:musigmaF}
\bmu([w],v) = \bmu([w^F],v) 
\quad \text{ for all }\quad   v\in \sigma_F \cap N.
\end{equation}
Indeed, if $v \in {\rm Int}(\sigma_F) \cap N$, then 
$\lim_{t \to 0} v(t) \cdot [w]= [w^F]$ and the formula holds.
On the other hand, if $v\in ( \sigma_F \setminus {\rm Int}(\sigma_F)) \cap N$, then $\lim_{t\to 0 }v(t) \cdot [w] = [w^G]$, where  $G$ is the face of $Q$  such that $v\in {\rm Int}(\sigma_G)$.
Using that any element in $N \cap {\rm Int}(\sigma_F)$ gives a degeneration $[w^G] \rightsquigarrow [w^F]$ and Lemma \ref{l:futnormfiberwise}, we see
\[
\bmu(z,[w]) = \bmu([w^G],v) = \bmu([w^F],v),\]
which shows \eqref{eq:musigmaF} holds.

Now, consider the  subspace $N_{\R}^F : = {\rm span}_{\R}(\sigma_F) \subset N_\R$ and the lattice $N^F : = N_{\R}^F \cap N$. 
Write $\T^{F}\subset \T$ for the subtorus satisfying $N_F= \Hom(\G_m,\T^F)$ and note that $\T^F$ fixes $[w^F]$. 
Applying Proposition \ref{p:minimizeroncone} to the log Fano pair corresponding to $[w^F]$ with the action by $\T^F$, we  see
$\inf \{ \bmu([w^F],v)\, \vert \,  v\in \sigma_F \cap N \}$ is a minimum, which completes the proof.
\end{proof}

\subsection{Existence of minimizers and constructibility results}

Using results from Sections \ref{ss:1psPGL} and \ref{ss:1psT},
we can deduce properties of $M^{\bmu}(X,D)$.

\begin{prop}\label{p:descending}
Assume Conjecture \ref{c:optdest} holds and set 
\[
\Gamma  := \left\{ M^{\bmu}(X,D) \, \vert \, [(X,D) \hookrightarrow \PP^m ] \in Z\right\} .
\]
Then the following hold:
\begin{enumerate}
    \item  The set $\Gamma$ is finite
    and $
Z_{ \geq {\bf m}}:=\{ [(X,D) \hookrightarrow \PP^m]  \in Z \, \vert \, M^{\bmu}(X,D) \geq {\bf m} \}$
is constructible  for each ${\bf m} \in \Gamma$.
    \item If $[(X,D)\hookrightarrow \PP^m] \in Z$ and $(X,D)$ is K-unstable, then the infimum $M^{\bmu}(X,D)$ is attained.
\end{enumerate}
\end{prop}

The proof is similar to an argument in \cite[\S 6]{BLZ19} that proves a related constructibility result for $M_{1}^{\bmu}$.

\begin{proof}
By Proposition \ref{p:mubehaviour}.1, there exists a decomposition $Z = \bigsqcup_{i=1}^s Z_i $ such that each $Z_i$ is a locally closed subset and functions $\bmu^{i}: N \to \R^2 \cup \{ ( + \infty, +\infty)\}$ so that
$\bmu^{i}(v) = \bmu(z,v)$ for each $z\in Z_i$ and $v\in N$ . 
Set 
${\bf m}^i :  = \inf_{v \in N} \bmu^{i}$.
Furthermore, Proposition \ref{p:mubehaviour}.2 guarantees the existence of 
$v_i\in N$ so that ${\bf m}^i : = \bmu^i (v^i)$ when 
 ${\bf m}^i<{\bf 0}$.

Now, fix $z= [(X,D) \hookrightarrow \PP^m]\in Z$ and recall that
\begin{equation}\label{e:constructible1}
M^{\bmu}(X,D) = \inf_{\la \in \Hom(\G_m , \PGL) } \bmu(z,\la)  
\end{equation}
 if  $(X,D)$ is K-unstable 
and the inequality $\leq $ holds if $(X,D)$ is K-semistable by Proposition \ref{p:M(X,D)=infl}.
To compute the right hand side infimum, note that 
\[
 \inf_{\la \in \Hom(\G_m , \PGL)}\bmu( z, \la ) = \inf_{g \in \PGL} \inf_{v\in N} \bmu(g\cdot z,v) 
\]
by \eqref{e:infmulav}
and
$\displaystyle \inf_{v} \bmu(g\cdot z,v) = \inf_{v } \bmu^j(v) = {\bf m^j}$ when $g \cdot z\in  Z_j $. 
Therefore, 
\begin{equation}\label{e:constructible2}
\inf_{\la \in \Hom(\G_m , \PGL)}\bmu( z, \la ) = \min \{ {\bf m}^j \,\vert \, \PGL \cdot z \in Z_j \} 
.\end{equation}
Combining \eqref{e:constructible1} and \eqref{e:constructible2} gives
$\Gamma \subseteq \{ {\bf 0}\} \cup \{ {\bf m}^1 , \ldots , {\bf m}^s \} $ and, hence, is finite.
In addition, for each ${\bf m } \in \Gamma$,
\[
Z_{\geq {\bf m} } =  Z \setminus \bigcup_{ {\bf m}^i < {\bf m} } \PGL\cdot Z_i .
\]
Since each set  $\PGL \cdot Z_i$ is constructible by Chevalley's Theorem,  $Z_{\geq {\bf m}}$ is also constructible. 

To see (2) holds, fix $z=[(X,D) \hookrightarrow \PP^m]\in Z$ with $(X,D)$ K-unstable. By the previous discussion, we may choose  $i\in \{1,\ldots, s\}$ and $g\in \PGL$ so that $M^{\bmu}(X,D) = {\bf m^i}$ and ${g\cdot z \in Z_i}$.
Since
$
{\bf m}^i=
\bmu(g\cdot z, v_{i})$, the action of $v_i$ on $g\cdot [(X,D) \hookrightarrow \PP^m]$ induces a special test configuration of $(X,D)$ computing $M^{\bmu}(X,D)$. 
\end{proof}

\begin{rem} 
While the result is stated for closed points of $Z$, the argument extends to geometric points. Specifically, morphisms $z':= \Spec(k') \to Z$, where $k'$ is an algebraically closed field, the conclusion of the proposition holds.  
\end{rem}

Using
 Proposition  \ref{p:descending}.2, we deduce the following statement.

 \begin{prop}\label{c:tfae}
 Let $(X,D)$ be a K-unstable log Fano pair. The following are equivalent: 
\begin{enumerate}
\item 
there exists a test configuration $(\cX,\cD)$ so that $M^{\bmu}(\cX,\cD) = \bmu(\cX,\cD)$;
\item  there exists a test configuration $(\cX,\cD)$ so that $M^{\bmu}_1(\cX,\cD) = \mu_1(\cX,\cD)$;
\item  the value $M_2^{\bmu}(X,D)$ is finite.
\end{enumerate}
 \end{prop}
 
 \begin{proof}
It is clear that (1) $\Rightarrow$ (2) and (2) $\Leftrightarrow (3)$.  By Proposition \ref{p:descending}.2, (2) $\Rightarrow$ (1) holds.
 \end{proof}

\begin{proof}[Proof of Theorem \ref{t:HNfilts}.1] Combining Theorem \ref{t:BLZinf}.1 and Proposition \ref{c:tfae} yields the  result.
\end{proof}


\section{Uniqueness of minimizers}\label{s:uniqueness}

In this section,  we prove Theorem \ref{t:HNfilts}.2  on the uniqueness of test configurations minimizing  $\bmu$ for K-unstable log Fano pairs. 
The proof uses a result from  \cite{BLZ19} stating that any two test configurations minimizing $\bmu$ can be connected via an equivariant family over $\A^2$ and properties of $\bmu$ described in Section \ref{s:logFanotorus}.

\subsection{Equivariant degenerations over affine two-space}

Let $(X,D)$ be a log Fano pair and $g:(\fX,\fD)\to \A^2$
 a family of log Fano pairs with the data of:
\begin{enumerate}
    \item an isomorphism $(\fX_{\bf 1} , \fD_{\bf 1}) \simeq (X,D)$, where ${\bf 1}: = (1,1) \in \A^2$, and
    \item  a $\G_m^2$-action on $(\fX, \fD)$ extending the standard  $\G_m^2$-action on $\A^2$.
    \end{enumerate}
The restrictions 
$ (\fX,\fD)_{ \A^1 \times 1} $
and 
$ 
(\fX,\fD)_{ 1\times \A^1} 
$ with
$\G_m$-actions given by $ \G_m \times 1 $ and $1\times  \G_m$
are naturally special test configurations of $(X,D)$, 
and we denote them  by $(\cX^1,\cD^1)$ and $(\cX^2,\cD^2)$.
Families of the above form appear in \cite{HL14,LWX18,BLZ19}.

The family $(\fX, \fD)$
can be recovered from the data of the test configurations $(\cX^1, \cD^1)$ and $(\cX^2,\cD^2)$. 
Indeed, fix an integer $r>0$ such that $\mathfrak{L}:=-r(K_{\fX/\A^2}+\fD)$ is a Cartier divisor. Set  $\cL^i:=-r(K_{\cX^i/\A^2}+\cD^i)$ and 
$L=-r(K_X+\Delta)$.
Write 
\[
R (X,L): = \bigoplus_{m \in \Z} R_m = \bigoplus_{m \in \Z}H^0(X, \cO_{X}(mL))\]
for the section ring of $L$. 
The test configurations $(\cX^1, \cD^1)$ and $(\cX^2, \cD^2)$  induce decreasing multiplicative $\Z$-filtrations $\cF^\bullet $ and $\cG^\bullet $ of $R(X,L)$ that recover the data of the test configuration; see \cite[Section 2.5]{BHJ17}.
In particular, there exists a  $\G_m$-equivariant isomorphism of graded $k[t]$-algebras
\[
\bigoplus_{m \in \N} H^0 (\cX^1,m \cL^1)
\simeq
\bigoplus_{m \in \N} \bigoplus_{p \in \Z} ( \cF^p R_m ) t^{-p}
,\]
where the $\G_m$-action on the right hand side algebra is induced by the  grading by $p$,
and a corresponding isomorphism holds for $\cX^2$ in terms of the filtration $\cG$.

\begin{lem}\label{l:familya2}
There is a $\G_m^2$-equivariant isomorphism of graded $k[x,y]$-algebras
\[
\bigoplus_{m \in \N} 
H^0 (\fX,m \fL)
\simeq 
 \bigoplus_{m \in \N} \bigoplus_{(p,q) \in \Z^2 } ( \cF^{p} R_m \cap \cG^q R_m ) x^{-p}y^{-q}
.
\]
Furthermore, $\bigoplus_{m \in \N} 
H^0 (\fX_{\bf 0},m \fL_{\bf 0})  \simeq  \bigoplus_{m } \bigoplus_{p,q} \gr^{p,q} V_m  $,
where \[
\gr^{p,q} R_m : = ( \cF^p R_m \cap \cG^q R_m )/ ( \cF^{p+1} R_m \cap \cG^q R_m  + \cF^{p} R_m \cap \cG^{q+1} R_m ).  \]
\end{lem}

\begin{proof}
Note that $g_*  \cO_{\fX}(m \fL)$ is a vector bundle on $\A^2$ and commutes with base change, since $H^i \left(\fX_t, \cO_{\fX_t}(m \fL_t) \right)= 0 $ for all $i>0$ and $t\in \A^2$ by Kawamata-Viehweg vanishing. 
Therefore,
\[
H^0 \left(\fX,  \cO_{\fX}(m\cL)\right) 
\simeq 
H^0 \left( \A^2,  g_*\cO_{\fX}(m\cL)\right)
\simeq 
H^0 \left(\A^2 \setminus 0, j^* g_* \cO_{\fX}(m\cL)\right) ,
\]
where  $j: \A^2 \setminus {\bf 0} \hookrightarrow \A^2$. 
To compute the latter module, consider the open immersions
\[
j_x: U_x \hookrightarrow \A^2,
\quad 
j_y: U_y 
\hookrightarrow \A^2,
\quad 
\text{ and }
\quad
j_{xy}: U_{xy} \hookrightarrow \A^2,\]
given by the loci where $x$, $y$, and $xy$ do not vanish. 
Using that 
$\fX\vert_{U_x} \simeq \cX_2 \times (\A^1\setminus 0)$,
 $\fX\vert_{U_y} \simeq \cX_1 \times (\A^1\setminus 0)$, 
 and $\fX_{U_{xy}} \simeq X\times (\A^1\setminus 0)^2$,
we see 
\begin{equation*}
H^0\left(U_x, j_x^* g_* 
\cO_{\fX}(m\fL)\right) \simeq  \bigoplus_{(p,q) \in \Z^2} (\cG^q R_m) x^{-p} y^{-q} 
\, \text{ and } \,
H^0 \left( U_y, j_y^* g_* 
\cO_{\fX}(m\fL) \right) \simeq  \bigoplus_{(p,q) \in \Z^2} (\cF^{p} R_m) x^{-p} y^{-q} 
,
\end{equation*}
and are both contained in 
\[
H^0 \left( U_{xy}, j_{xy}^* g_* \cO_{\fX}(m\fL) \right) \simeq \bigoplus_{(p,q) \in \Z^2} R_m x^{-p}y^{-q}.
\]
Using that  $\A^2 \setminus {\bf 0}= U_x \cup U_y$, we conclude that the first statement holds.

For the second statement, note that
$ H^0 (\fX_0, m\fL_0) \simeq 
 H^0 (\fX, m\fL)  \otimes  k[x,y]/(x,y)$.
Since
\[
(x,y) \big(  \bigoplus_{p,q} (\cF^p R_m \cap \cG^q R_m) x^{-p} y^{-q} \big)= 
 \bigoplus_{p,q} ( \cF^{p+1} R_m \cap \cG^q R_m + 
\cF^p R_m \cap \cG^{q+1} R_m )x^{-p} y^{-q}
,\]
the second statements holds.
\end{proof}

The following proposition is a special  case of a statement regarding maps from $[\A^2/ \G_m^2]$ to  algebraic stacks \cite[Lemma 4.23]{HL14}. We provide a proof using the previous lemma.

\begin{prop}\label{p:kernel}
If $\ker(\G_m^2 \to \Aut( \fX_{\bf 0}, \fD_{\bf 0}))$ 
contains $\{ (t,t^{-1}) \, \vert \, t\in \G_m\}$, then
$(\cX^1,\cD^1)$ and $(\cX^2,\cD^2)$
are isomorphic as test configurations.
\end{prop}

\begin{proof}
Let $\rho:\G_m \to \G_m^2$ denote the 1-PS defined by $t\mapsto(t,t^{-1})$.
By assumption, $\rho$
acts trivially on $(\cX_{\bf 0 },\cD_{\bf 0 })$ and, hence,  acts trivially on $H^0 (\fX_{\bf 0},m \fL_{\bf 0})$, which is isomorphic to $\oplus_{p,q} \gr^{p,q}R_m$ by Lemma \ref{l:familya2}.
Since $\rho$ acts with weight  $p-q$ on $\gr^{p,q} R_m$, this means $\gr^{p,q} R_m = 0$ if $p-q\neq 0$.

The latter implies the filtrations $\cF^\bullet$ and $\cG^\bullet$ of $R_m$
are equal. Indeed, by \cite[Lemma 3.1]{AZ20}, there exists a basis $\{s_1,\ldots, s_{N_m} \}$ of $R_m$ 
such that 
\[
\cF^p R_m = {\rm span} \langle s_i \, \vert \, \ord_{\cF}(s_i) \geq p \rangle\quad \text{ and } \quad 
\cG^q R_m = {\rm span} \langle  s_i \, \vert \, \ord_{\cG}(s_i) \geq q\rangle,\]
where $\ord_{\cF}(s_i): = \max\{ p \, \vert \, s_i \in \cF^p R_m \}$ and  $\ord_{\cG}(s_i) := \max\{ q \vert s_i \in \cG^q R_m \}$. 
Since $\gr^{p,q} R_m $ has basis given by 
$\{ \overline{s_i} \, \vert \, \ord_{\cF}(s_i) =p  \text{ and }\ord_{\cG}(s_i)=q \}$, the  vanishing of $\gr^{p,q}R_m$ for $p\neq q$ implies $\ord_{\cF}(s_i) = \ord_{\cG}(s_i)$ for each $i$ and, hence, $\cF^\bullet  = \cG^\bullet $.
Therefore,  $(\cX^1,\cD^1)$ and $(\cX^2,\cD^2)$ are isomorphic as test configurations.
\end{proof}

\subsection{Proof of Theorem \ref{t:HNfilts}.2}

The proof of  Theorem \ref{t:HNfilts}.2
uses the following result that builds on earlier work from  \cite[Proof of Theorem 3.2]{LWX18}).

\begin{thm}[{\cite[Theorem 5.4]{BLZ19}}]
\label{t:BLZ}
Let $(\cX^1,\cD^1)$ and $(\cX^2,\cD^2)$ be test configurations of a log Fano pair $(X,D)$ that is K-unstable.
If
\begin{equation}\label{eq:mu1=m}
\mu_1 (\cX^1, \cD^1) = M_1^{\bmu}(X,D) =\mu_1(\cX^2,\cD^2)
,\end{equation}
then there exists  $\G_m^2$-equivariant family of log Fano pairs $g: (\fX,\fD) \to \A^2$ with the data of an isomorphism $(\fX_{ {\bf 1}},\fD_{ {\bf 1}}) \simeq (X,D)$ such that their are isomorphisms of test configurations
\[
(\cX^1,\cD^1) \simeq (\fX,\fD)_{1\times \A^1}
\quad \text{ and } \quad (\cX^2,\cD^2) \simeq (\fX,\fD)_{1\times \A^1}.\]
\end{thm}

\begin{proof}[Proof of Theorem \ref{t:HNfilts}.2]
Let
$(\cX^1, \cD^1)$ and $(\cX^{2},\cD^{2})$ be special test configurations of a K-unstable log Fano pair $(X,D)$ satisfying
\[
\bmu (\cX^{1}, \cD^{1})=M^{\bmu}(X,D)=\bmu(\cX^{2}, \cD^{2}) 
.\]
Since $(X,D)$ is K-unstable, $\Fut(\cX^i,\cD^i)<0$ for $i=1,2$.
Therefore, we may scale $(\cX^1, \cD^1)$ and $(\cX^{2},\cD^{2})$
such that
$\Fut(\cX^1,\cD^1) =F= \Fut(\cX^2,\cD^2)$
for some $F<0$.

Using that $\R^2$ is endowed with the lexicographic order,   \eqref{eq:mu1=m}
holds.
Let  ${(\fX ,\fD) \to \mathbb{A}^2}$ 
denote the $\G_m^2$-equivariant family of log Fano pairs satisfying the conclusion of
Theorem \ref{t:BLZ}. Consider the induced $\G_m^2$-action on $(\fX,\fD)_{{\bf 0}}$
and the functions $\Fut(\,) $ and $\bmu(\, )$ on
$ N_{\R}\simeq \R^2$, where $N:=\Hom( \G_m, \G_m^2)$.
Note that
\[
\bmu( 1,0 )
= 
\bmu(\cX^{1},\cD^{1})
\quad 
\text{ and }
\quad 
\bmu( 0,1 )
= 
\bmu(\cX^{2},\cD^{2}),
\]
which are equal to $M^{\bmu}(X,D)$ by assumption. 
Additionally,
$\bmu ( a,b ) \geq M^{\bmu}(X,D)$ for all $(a,b)\in \Z_{\geq 0}^2$,
since pulling back $(\fX,\fD)\to \A^2$ via the map $\A^1 \to \A^2$ sending $t\mapsto (t^a,t^b)$ induces a test configurations $(\cX^{(a,b)}, \cD^{(a,b)})$ of $(X,D)$ and 
$ \bmu (a,b) = \bmu(\cX^{(a,b)}, \cD^{(a,b)}) \geq M^{\bmu}(X,D)$. Therefore, \[
\bmu: \R_{\geq 0}^2 \cap (\N^2 \setminus (0,0) ) \to \R^2\]
is minimized at both $(1,0)$ and $(0,1)$. 
The previous statement combined with Proposition \ref{p:minimizeroncone} implies that
$\G_m^2 \to \Aut(X,D)$ has a positive dimensional kernel. Therefore, there exists $(0,0)\neq (a,b) \in \Z^2 $ such that the 1-PS $ \G_m \to \G_m^2$ defined by $t\mapsto (t^{a},t^b)$ acts trivially on $(\fX,\fD)_{{\bf 0}}$. 
Since
 \[
0=\Fut(a,b ) = a \Fut(1,0) +b \Fut(0,1) =  aF+bF,
 \]
 where the first inequality uses that the action is trivial and the second is the linearity of $\Fut$,
 we see $a=-b$ and, hence,
$\{ (t, t^{-1}) \, \vert \, t\in \G_m\} \subset \ker ( \G_m^2 \to \Aut(X,D))$. 
Applying  Proposition \ref{p:kernel}, we conclude $(\cX^1,\cD^1)\simeq (\cX^2,\cD^2)$.
\end{proof}

\subsection{Equivariant test configurations}

The uniqueness of minimizers (Theorem \ref{t:HNfilts}.2) has the following immediate consequence. 

\begin{cor}\label{c:equiv}
Let $(X,D)$ be a log Fano pair with the action of a group $G$. 
If $(X,D)$ is K-unstable and $(\cX,\cD)$ computes $M^{\bmu}(X,D)$, then 
$(\cX,\cD)$ is $G$-equivariant.
\end{cor}

Recall, a special test configuration $(\cX,\cD)$ is \emph{$G$-equivariant} if the
 induced $G$-action on $(\cX,\cD)_1$ extends to a $G$-action on $(\cX,\cD)$ that commutes with the $\G_m$-action. 
For an alternate characterization, note that, for $g\in G$,
postcomposing the isomorphism $(\cX,\cD)_{1} \simeq (X,D)$ by $g^{-1}:(X,D) \to   (X,D)$ induces a new  test configuration of $(X,D)$, which we denote by $(\cX_g,\cD_g)$. 
The  test configuration $(\cX,\cD)$ is $G$-equivariant if and only if $(\cX_g,\cD_g)\simeq (\cX,\cD)$ as test configurations for all $g\in G$.

\begin{proof}
Since $(\cX,\cD)$ computes $M^{\bmu}(X,D)$, $(\cX_g,\cD_g)$ also computes $M^{\bmu}(X,D)$ for any $g\in G$. Using that $\Fut(\cX,\cD)= \Fut(\cX_g,\cD_g)<0$, the proof of Theorem \ref{t:HNfilts}.2 implies $(\cX,\cD)$ and  $(\cX_g,\cD_g)$ are isomorphic as test configuration.
\end{proof}

\subsection{Behaviour under field extension}
Let $K'/K$ be an extension of  characteristic zero fields 
with $K'$ algebraically closed. Let $(X,D)$ be a log Fano pair defined over $K$, and set $(X',D')=(X_{K'},D_{K'})$. 

We seek to compare  $M^{\bmu}$ of the two log Fano pairs. (We note that while $M^{\bmu}(X,D)$ of log Fano pair  was defined in Section \ref{ss:bmu} when the base field is algebraically closed, the definition extends verbatim.)
Since any special test configuration of $(X,D)$ induces a special test configuration of $(X',D')$ via field extension, it is clear that
\begin{equation}\label{eq:K'/K}
M^{\bmu}(X',D') \leq M^{\bmu}(X,D).
\end{equation}

\begin{prop}\label{p:fieldextend}
The equality $M^{\bmu}(X',D')= M^{\bmu}(X,D)$ holds.
\end{prop}

The proposition is a consequence of the following lemma and \cite{Zhu20}.

\begin{lem}\label{l:K'/K}
Assume $(X',D')$ is K-unstable. 
If a special test configuration $(\cX',\cD')$ computes  $M^{\bmu}(X',D')$, then it descends to  a special test configuration of $(X,D)$.
\end{lem}

\begin{proof}
The proof uses the correspondence  between test configurations and finitely generated filtrations of the section ring \cite[Section 2.5]{BHJ17}, which also holds when the base field is not algebraically closed.
Fix  $r>0$ so that $L:= -r(K_X+D)$ and $L':= -r(K_{X'}+D')$ are Cartier divisors, 
and write $R(X,L)$ and $ R(X',L')$ for the section rings of $L$ and $L'$. Note that $R(X,L)\otimes K'\simeq R(X',L')$. 

The  test configuration $(\cX',\cD')$ induces a finitely generated $\Z$-filtration $\cF'^\bullet$ of $R(X',L')$.
Since $(X',D')$ is K-unstable and $(\cX',\cD')$ computes $M^{\bmu}(X',D')$, 
the proof of Corollary \ref{c:equiv} implies $(\cX',\cD')$ is invariant with respect to ${\rm Gal}(K'/K)$. Therefore, the filtration $\cF'^{\bullet}$ is invariant with respect to the action of $\Gal(K'/K)$ on $R(X',L')$. 
The latter implies there exists a finitely generated $\Z$-filtration $\cF^{\bullet}$ of $R(X,L)$ so that  $\cF'^{\bullet} \simeq \cF^\bullet \otimes K'$ (see e.g. the proof of \cite[Proposition 3]{Lan75}). Hence,  $\cF^{\bullet}$ induces a special test configuration $(\cX,\cD)$ of $(X,D)$ so that $(\cX',\cD')\simeq (\cX,\cD) \times \Spec(K)$.
\end{proof}

\begin{proof}[Proof of Proposition \ref{p:fieldextend}]
If $(X',D')$ is K-semistable, then $(X,D)$  is K-semistable by \eqref{eq:K'/K}.
Hence, $M^{\bmu}(X,D) = {\bf 0} = M^{\bmu}(X,D)$.

Now assume $(X',D')$ is K-unstable.
By  \cite[Theorem 1.2]{Zhu20},  $\delta(X',D') = \delta(X,D)$.
Using Theorem \ref{t:BLZinf}.1, we see  $M_1^{\bmu}(X',D') = M_1^{\bmu}(X,D)$.
If $M_2^{\bmu}(X',D')=+\infty $, then 
$M_2^{\bmu}(X,D) = +\infty $ as well by \eqref{eq:K'/K}. 
If $M_2^{\bmu}(X',D')<+\infty$, then Proposition \ref{c:tfae}
guarantees  the existence of a special test configuration $(\cX',\cD')$ computing $M^{\bmu}(X',D')$. 
By Lemma \ref{l:K'/K}, $(\cX',\cD')$ descends to a special test configuration $(\cX,\cD)$ of $(X,D)$. 
Since
\[
M^{\bmu}(X',D') = \bmu(\cX',\cD')= \bmu(\cX,\cD) \geq M^{\bmu}(X,D)
\]
and 
 $M^{\bmu}(X',D') \leq M^{\bmu}(X,D)$ by \eqref{eq:K'/K}, we conclude 
 $M_2^{\bmu} (X',D') =M_2^{\bmu}(X,D)$.
 \end{proof}

\section{Behaviour in families}\label{s:lsc}

In this section we describe the behaviour $M^{\bmu}$ under specialization.

\subsection{Semicontinuity under specialization}
Let $R$ be a DVR essentially of finite type over $k$. Write $K$ and $\kappa$ for the fraction field and residue field of $R$. 
 
\begin{prop}\label{p:lscDVR}
If $(X,D)\to \Spec(R)$ is a family of log Fano pairs, then
\[
M^{\bmu}(X_{K},D_{K}) \geq M^{\bmu}(X_{\kappa}, D_{\kappa}). 
\]
\end{prop}

The statement follows from combining the lower semi-continuity of the $\delta$-invariant \cite{BLX19}
with the following result from \cite{BLZ19} on extending  test configurations.

\begin{prop}\label{t:BLZTheta}
Let $(X,D)\to \Spec(R)$ be a family of log Fano pairs. If $(\cX_K,\cD_K)$ is a test configuration of $(X_K,D_K)$ such that 
\[ 
\mu_1(\cX_K,\cD_K) \leq M_{1}^{\bmu}(X_\kappa,D_\kappa) \quad \text{ and } 
\quad 
\mu_1(\cX_K,\cD_K)\leq  0,\]
then the test configuration extends to a $\G_m$-equivariant family of a log Fano pairs $(\cX,\cD)\to \A^1_{R}$ with the data of an isomorphism $(\cX,\cD)_{1} \simeq (X,D)$ over $\Spec(R)$.
\end{prop}

\begin{proof}
When $R$ is the germ of a smooth curve, this  is \cite[Theorem 5.2]{BLZ19} with $\mu:= \mu_1(\cX_K,\cD_K)$. The argument in \emph{loc. cit.} extends verbatim to the case when $R$ is a DVR essentially of finite type over $k$.
\end{proof}

\begin{proof}[Proof of Proposition \ref{p:lscDVR}]
If $(X_K,D_K)$ is K-semistable, then the statement holds trivially.
Now, assume  $(X_K,D_K)$ is K-unstable.
By \cite[Corollary 1.2]{BLX19},  $(X_\kappa,D_\kappa)$ is also K-unstable.
Using Theorem \ref{t:BLZinf} and Proposition \ref{p:fieldextend}, we see
\[
M_1^{\bmu}(X_{K}, D_{{K}})= \delta(X_{\overline{K}}, D_{\overline{K}}) -1 
\quad \text{ and } \quad 
M_1^{\bmu}(X_{{\kappa }}, D_{{\kappa }})= \delta(X_{\overline{\kappa }}, D_{\overline{\kappa }})-1.\]
Since $\delta$ is lower semicontinuous along geometric fibers \cite{BL18}, 
$M_1^{\bmu}(X_{K}, D_{{K}})\geq
 M_1^{\bmu}(X_{{\kappa }}, D_{{\kappa }})$.
 If the inequality  is strict, then $M^{\bmu}(X_K,D_K)> M^{\bmu}(X_\kappa, D_\kappa)$. 
 
We are left to consider the case when $M_1^{\bmu}(X_{K}, D_{{K}})=
 M_1^{\bmu}(X_{{\kappa }}, D_{{\kappa }})$. 
 First assume there exits a test configuration $(\cX_K,\cD_K)$ of $(X_K,D_K)$
such that
 $ M^{\bmu}(X_K,D_K)=\bmu( \cX_K,\cD_K) $.
 By Proposition \ref{t:BLZTheta},  the test configuration extends to a $\G_m$-equivariant family of a log Fano pairs $(\cX,\cD)\to \A^1_{R}$ with the data of an isomorphism $(\cX,\cD)_{1} \simeq (X,D)$ over $\Spec(R)$. Therefore,
 \[
M^{\bmu}(X_K,D_K) = \bmu  (\cX_K,\cD_K)  = \bmu(\cX_\kappa, \cD_\kappa) \geq 
M^{\bmu}(X_\kappa ,D_\kappa),
 \]
 where the second equality is Lemma \ref{l:futnormfiberwise}.
  Next, assume $M^{\bmu}(X_{K},D_{K})$ is not achieved by a test configuration. By Proposition \ref{c:tfae}, $M_2^{\bmu}(X_K,D_K)=+\infty$. 
  Therefore,  $M^{\bmu}(X_{K},D_{K})\geq M^{\bmu}(X_\kappa,D_\kappa)$ also holds.
\end{proof}

\subsection{Specialization of minimizers}

\begin{prop}\label{p:minimum-same}
If $(X,D)$ is a log Fano pair that is K-unstable and $(\cX,\cD)$ is a special test configuration of $(X,D)$ computing $M^{\bmu}(X,D)$, then 
\[M^{\bmu}(X,D)= M^{\bmu}(\cX_0,\cD_0).\]
\end{prop}

The proof is similar to that of \cite[Lemma 3.1]{LWX18}.

\begin{proof}
Let $\xi$ denote the $\G_m$-action on $(\cX_0,\cD_0)$
and obverse that 
\[M^{\bmu}(X,D) =\bmu(\cX,\cD)= \bmu(\cX_0,\cD_0; \xi) \geq M^{\bmu}(\cX_0,\cD_0).
\]
To prove the reverse inequality,
 suppose for sake of contradiction that $M^{\bmu}(X,D) >M^{\bmu}(\cX_0,\cD_0)$. 

Observe that
$M^{\bmu}(\cX_0,\cD_0)$ is computed. 
Indeed, 
since
$M_1^{\bmu}(X,D) = M_1^{\bmu}(\cX_0,\cD_0)$
(Theorem \ref{t:BLZinf})
and $ M_1^{\bmu}(X,D) =\mu_1(\cX_0,\cD_0 ;\xi)$, 
$M_1^{\bmu}(\cX_0,\cD_0)$ is computed by the product test configuration of $(\cX_0,\cD_0)$ induced by $\xi$.
Therefore,  Proposition \ref{c:tfae} implies the existence  of a test configuration $(\cX',\cD')$ computing $M^{\bmu}(\cX_0,\cD_0)$.

Let $\xi'$ denote the $\G_m$-action on $(\cX'_0,\cD'_0)$.
Since $(\cX',\cD')$ is equivariant 
 with respect to the $\G_m$-action $\xi$  on $(\cX_0,\cD_0)$ (Corollary \ref{c:equiv}), $\xi$ induces a $\G_m$-action on $(\cX'_0,\cD'_0)$ commuting with $\xi'$  and we somewhat abusively denote it by $\xi$.
 
By the proof of \cite[Lemma 3.1]{LWX18}, for $k\gg0$ , there exists a test configuration $(\cX'',\cD'')$ of $(X,D)$ such that $(\cX''_0,\cD''_0)$ is isomorphic to $(\cX'_0,\cD'_0)$ and has $\G_m$-action given by $k\xi+\xi'$. 
To compute $\bmu(\cX'',\cD'')$, 
first note that 
\[
\bmu(\cX'_0,\cD'_0 ; \xi ) = \bmu(\cX_0,\cD_0; \xi) = M^{\bmu}(X,D) 
>
M^{\bmu}(\cX_0,\cD_0) =\bmu(\cX'_0,\cD'_0 ; \xi' )
,\]
where the first equality is by Lemma \ref{l:futnormfiberwise}.
By Proposition \ref{p:quasiconvex}, we see \[
\bmu(\cX'_0,\cD'_0; k \xi +  \xi' )
<  \max \left\{ \bmu( \cX'_0; \cD'_0 ;\xi ) , 
\bmu( \cX'_0; \cD'_0 ;  \xi'  ) \right\} = M^{\bmu}(X,D).\]
Since $\bmu(\cX'',\cD'') = \bmu(\cX_0,\cD_0; k\xi + \xi')
$, this gives
$\bmu(\cX'',\cD'') < M^{\bmu}(X,D)$, which is absurd.
\end{proof}

\section{Existence of $\Theta$-stratification and properness}\label{s:theta}

In this section, we will prove Theorem \ref{t-theta}, that is, the existence of $\Theta$-stratification on $\cM_{n,V,c}^{\rm Fano}$ under the assumption of Conjecture \ref{c:optdest}. By \eqref{eq:hilbert}, we know that $\cM_{n,V,c}^{\rm Fano}=\sqcup_{h}\cM_{h,c}^{\rm Fano}$ where $h$ runs over all Hilbert functions. We will construct a $\Theta$-stratification on each $\cM_{h,c}^{\rm Fano}$.

For simplicity, we denote by $\sX:=\cM_{h,c}^{\rm Fano}$. 
Let $\Gamma:=\{M^{\bmu}(X,D)\mid [(X,D)]\in |\cM_{h,c}^{\rm Fano}|\}$ be the subset of $(\R\cup\{\pm\infty\})^2$ equipped with the lexicographic order.
For each $\bfm\in \Gamma$, we define  the subfunctor $\sX_{\geq \bfm}$ of $\sX$ as 
\[
 \sX_{\geq \bfm}(T) = \{ [(X,D) \to T] \in \  \sX(T) \, \vert \, M^{\bmu}(X_{\overline{t}},D_{\overline{t}}) \geq \bfm \quad \text{for all }\, t\in T \}.
\]
It is clear that $\sX_{\geq\mathbf{0}}=\cM_{h,c}^{\rm Kss}$.

\begin{prop}
Assume Conjecture \ref{c:optdest} holds. 
For each $\bfm\in \Gamma$, the functor $\sX_{\geq \bfm}$ is represented by an open substack of $\sX$ of finite type.
\end{prop}

\begin{proof}
Let $\bfm:=(\bfm_1,\bfm_2)\in\Gamma$. Then we know that every log Fano pair $(X,D)$ with $M^{\bmu}(X,D)\geq \bfm$ must satisfy $\delta(X,D)-1\geq \bfm_1$. Hence, $\sX_{\geq \bfm}$ is a subfunctor of $\cM_{h,c}^{\delta\geq \bfm_1+1}$ which is a finite type open substack of $\cM_{h,c}^{\rm Fano}$ by Theorem \ref{t:bounded} and \cite[Theorem 1.1]{BLX19}. Thus, it suffices to show that $\sX_{\geq \bfm}$ is an open substack of $\cM_{h,c}^{\delta\geq \bfm_1+1}$. By Section \ref{ss:paramtest}, we know that $\cM_{h,c}^{\delta\geq \bfm_1+1}\cong [Z/\PGL_{m+1}]$ where $m=h(r)-1$ and  $Z$ is a locally closed subscheme of a  scheme $\mathbf{P}$ of finite type which parametrizes certain K-flat relative Mumford divisors over a suitable Hilbert scheme.
By constructibility and lower semicontinuity of $M^{\bmu}$ from Propositions \ref{p:descending} and \ref{p:lscDVR}, we know that the locus $Z_{\geq \bfm}:=\{[(X,D)]\in Z\mid M^{\bmu}(X,D)\geq \bfm\}$ is an open subscheme of $Z$. Hence the proof follows.
\end{proof}

Next we construct the $\Theta$-stratum $\sS_{\bfm}\subset \uMap(\Theta, \sX_{\geq \bfm})$. 
First of all, we may write 
$\sX_{\geq \bfm}=[Z_{\geq \bfm}/\PGL_{m+1}]$. Let $\T\cong\G_m^{m}$ be a maximal torus of $\PGL_{m+1}$. Denote by $N:=\Hom (\G_m, \T)$ and $M:=\Hom(\T, \G_m)$. Let $N'\subset N$ be a subset representing conjugacy classes of $1$-PS' in $\PGL_{m+1}$ (e.g. $N'$ consists of 
$1$-PS' with ascending weights under an identification $N\xrightarrow{\cong} \Z^{m}$).   Then we know that 
\[
\uMap(\Theta,\sX_{\geq \bfm})= \uMap(\Theta,[Z_{\geq \bfm}/\PGL_{m+1}])=\bigsqcup_{\lambda\in N'} [Y_\lambda/P_\lambda],
\]
where $Y_\lambda$ is the union of Bialynicki-Birula strata of $Z_{\geq \bfm}$ associated to $\lambda$, and $P_{\lambda}:= \{g\in \PGL_{m+1}\mid \lim_{t\to 0} \lambda(t)g\lambda(t)^{-1} \textrm{ exists}\}$. We know that $Y_\lambda\to Z_{\geq \bfm}$ is a locally closed immersion with image $\{z\in Z_{\geq \bfm}\mid \lim_{t\to 0} \lambda(t)\cdot z\textrm{ exists}\}$. We will often identify a point in $Y_\lambda$ with its image in $Z_{\geq \bfm}$. In particular, if $\lambda=0\in N'$ then $Y_{0}=Z_{\geq \bfm}$ and $P_{0}=\PGL_{m+1}$, i.e. $[Y_0/P_0]=\sX_{\geq \bfm}$ is the connected component of $\uMap(\Theta,\sX_{\geq \bfm})$ parametrizing trivial maps $\Theta_k\to \Spec~k\to \sX_{\geq \bfm}$.
Thus to construct the $\Theta$-stratum $\sS_{\bfm}$, it suffices to find a suitable union of  connected components of $Y_\lambda$ for each $\lambda\in N'$. 

Suppose $\bfm\neq \mathbf{0}$. For each $\lambda\in N'\setminus \{0\}$, consider the subset $S_\lambda\subset Y_{\lambda}$ as 
\[
S_{\lambda} := \{z\in Y_{\lambda}\mid \bmu(z,\lambda)=\bfm\}.
\]
For $\lambda=0\in N'$, we define $S_0:= Y_0$. We will show that $S_{\lambda}$ is a disjoint union of connected components of $Y_{\lambda}$. Indeed, by the definition of $Y_\lambda$ (see \cite[Section 1.4]{HL14}) there is a $\G_m$-equivariant map $\phi_\lambda:Y_\lambda\times\A^1\to Z_{\geq \bfm}$ where the $\G_m$-action on $Z_{\geq \bfm}$ is $\lambda$ and $\phi_{\lambda}(z,1)=z$. Thus pulling back the universal log Fano family over $Z_{\geq \bfm}$ to $Y_\lambda\times\A^1$ under $\phi_\lambda$ and applying Lemma \ref{l:futnormfiberwise}, we see that $z\mapsto \bmu(z,\lambda)$ is a locally constant function on $Y_\lambda$. Hence $S_\lambda$ is a disjoint union of connected components of $Y_\lambda$.



\begin{prop}\label{p:thetaclosed}
Assume Conjecture \ref{c:optdest} holds. 
With the above notation, for any $\bfm\neq \mathbf{0}$ and $\lambda\in N'\setminus\{0\}$ the map $\mathrm{ev}_1(\phi_\lambda):S_\lambda \to Z_{\geq \bfm}$ is a closed immersion.
\end{prop}

\begin{proof}
By definition we know that $\mathrm{ev}_1(\phi_\lambda)$ is a locally closed immersion. Thus it suffices to show that it is proper.
Suppose $f:\Spec(R)\to Z_{\geq \bfm}$ is a morphism from a DVR such that $z_{K}:=f(\Spec (K))\in S_{\lambda}$. Let $(X,D)\to \Spec(R)$ be the $f$-pullback of the universal log Fano family over $Z_{\geq \bfm}$. Then $\lambda$ induces a special test configuration $(\cX_K,\cD_K)$ of $(X_K, D_K)$ such that $\bmu(\cX_K,\cD_K)=\bfm$.
 Since 
$z_\kappa =[(X_\kappa ,D_\kappa)]\in Z_{\geq \bfm}$, we know that $M^{\bmu}(X_\kappa ,D_\kappa)\geq \bfm$.
Hence Proposition \ref{t:BLZTheta} implies that the test configuration $(\cX_K,\cD_K)$ extends to a $\G_m$-equivariant family of log Fano pairs $(\cX,\cD)\to\A_R^1$ such that $(\cX,\cD)_1 \cong (X,D)$ over $\Spec(R)$. Moreover, Lemma \ref{l:futnormfiberwise} implies that the special test configuration $(\cX_\kappa,\cD_\kappa)$ over $\A_{\kappa}^1$ satisfies 
\[
\bmu(\cX_\kappa,\cD_\kappa)=\bmu(\cX_K,\cD_K)=\bfm.
\]
Hence we may extend $f$ to $\tilde{f}: \A_R^1\to Z_{\geq \bfm}$ where $\tilde{f}_t=\lambda(t)\cdot f$ for $t\in\G_m$ where $(\cX,\cD)$ is $\G_m$-equivariantly isomorphic to the $\tilde{f}$-pullback of the universal log Fano family over $Z_{\geq \bfm}$. Therefore, $f$ admits a lifting to $S_\lambda$ which implies that $\mathrm{ev}_1(\phi_\lambda):S_\lambda\to Z_{\geq \bfm}$ is proper.
\end{proof}

Denote by $N'_{\prim}$ the subset of $N'\setminus\{0\}$ consisting of primitive $1$-PS'. For $\bfm\neq \mathbf{0}$ we define 
\[
\sS_{\bfm}:= \bigsqcup_{\lambda\in N'_{\prim}}[S_\lambda/P_\lambda].
\]
For $\bfm=\mathbf{0}$ we define $\sS_{\mathbf{0}}:=[Y_0/P_0]=\sX_{\geq \mathbf{0}}$ parametrizing trivial maps.

\begin{thm}\label{t:sXtheta}
Assume Conjecture \ref{c:optdest} holds. 
Then the data $(\sX_{\geq \bfm}, \sS_{\bfm})_{\bfm\in\Gamma}$ form a well-ordered $\Theta$-stratification of $\sX=\cM_{h,c}^{\mathrm{Fano}}$.
\end{thm}

\begin{proof}
We first show that for each $\bfm\in\Gamma$, the stack $\sS_{\bfm}$ is a $\Theta$-stratum of $\sX_{\geq \bfm}$. The statement is clear when $\bfm=\mathbf{0}$ as $\sS_{\mathbf{0}}=\sX_{\mathbf{0}}$. Hence we may assume that $\bfm\neq \mathbf{0}$. By Proposition \ref{p:thetaclosed} we know that $S_\lambda\to Z_{\geq \bfm}$ is a closed immersion. Thus we know that the morphism $\mathrm{ev}_1: \sS_{\bfm}\to \sX_{\geq \bfm}$ is a composition of proper morphisms as below:
\[
\sS_{\bfm}=\sqcup_{\lambda}[S_\lambda/P_\lambda]\to [Z_{\geq \bfm}/P_\lambda] \to [Z_{\geq \bfm}/\PGL_{m+1}]=\sX_{\geq \bfm}.
\]
Hence $\mathrm{ev}_1$ is proper. Next, we show that $\mathrm{ev}_1$ is universally injective. Since we work over characteristic zero, it suffices to show that the $\PGL_{m+1}$-equivariant morphism 
\[
\psi:\PGL_{m+1}\times_{P_{\lambda}}S_\lambda\to Z_{\geq \bfm}
\]
is injective whose $\PGL_{m+1}$-quotient gives $\mathrm{ev}_1$. Suppose $(g_1,z_1)$ and $(g_2,z_2)$ in $\PGL_{m+1}\times S_\lambda$ have the same image in $Z_{\geq\bfm}$, i.e. $z_1=g_1^{-1} g_2 \cdot z_2$. Hence we know that $z_1$ and $z_2$ belong to the same $\PGL_{m+1}$-orbit in $Z_{\geq \bfm}$. In other words, they correspond to different embeddings into $\bP^m$ of the same log Fano pair $(X,D)$ with $M^{\bmu}(X,D)=\bfm$. Since $z_1,z_2\in S_\lambda$, we know that $\bmu(z_1,\lambda)=\bmu(z_2,\lambda)=\bfm$ which implies that $\lambda$ induces $\bmu$-minimizing primitive special test configurations $(\cX^1,\cD^1)$ and $(\cX^2,\cD^2)$ of  $(X,D)$. 
By uniqueness of minimizers from Theorem \ref{t:HNfilts}.2, we know that $(\cX^1,\cD^1)\cong (\cX^2,\cD^2)$ as test configurations. In other words, the two morphisms $\Theta\to \sX_{\geq \bfm}$ induced by $(z_i, \lambda)$ for $i=1,2$ represent the same point in the mapping stack. Therefore, we have that $z_2=p \cdot z_1$ for some $p\in P_\lambda$. Denote by $g:= g_1^{-1} g_2 p$, so that $z_1$ is a $g$-fixed point. By Corollary \ref{c:equiv}, we know that $g$ acts on the special test configuration $(\cX_1,\cD_1)$ which implies that $g\in P_\lambda$. In particular, $g_1^{-1}g_2\in P_\lambda$. Hence $\psi$ is injective which implies that $\mathrm{ev}_1$ is universally injective. As a result, we have shown that $\sS_{\bfm}$ is a weak $\Theta$-stratum of $\sX_{\geq \bfm}$ (see \cite[Definition 2.1]{HL14}). Since we work over characteristic zero, the weak $\Theta$-stratum $\sS_{\bfm}$ is also a $\Theta$-stratum of $\sX_{\geq \bfm}$ by \cite[Corollary 2.6.1]{HL14}.


Next, we show that the complement of $\sS_{\bfm}$ in $\sX_{\geq \bfm}$ is precisely $\sX_{>\bfm}$. This is trivial for $\bfm=\mathbf{0}$, so we assume $\bfm\neq \mathbf{0}$.
If $z=[(X,D)]\in S_{\lambda}$, then we have $M^{\bmu}(X,D)=\bmu(z,\lambda)=\bfm$. Hence $\sS_{\bfm}$ is disjoint from $\sX_{>\bfm}$. On the other hand, if $[(X,D)]\in |\sX_{\geq \bfm}|\setminus |\sX_{>\bfm}|$, then Theorem \ref{t:HNfilts}.1 implies that there exists a primitive special test configuration $(\cX,\cD)$ of $(X,D)$ such that $M^{\bmu}(X,D)=\bmu(\cX,\cD)=\bfm$.
By Proposition \ref{p:minimum-same} we know that $M^{\bmu}(\cX_0,\cD_0)=M^{\bmu}(X,D)=\bfm$. Hence the test configuration $(\cX,\cD)$ corresponds to a point in $\uMap(\Theta, \sX_{\geq m})$ with $\bmu(\cX,\cD)=\bfm$. From the definition of $S_\lambda$ and $\sS_{\bfm}$, we know that $(\cX,\cD)$ is induced by some $\lambda\in N'_{\prim}$ and $z= [(X,D)\hookrightarrow\PP^m]\in S_\lambda$.
 Hence 
 $[(X,D)]$ belongs to the image of $\mathrm{ev}_1:\sS_{\bfm}\to \sX_{\geq \bfm}$. This shows that the complement of $\sS_{\bfm}$ in $\sX_{\geq \bfm}$ is $\sX_{>\bfm}$. 

Finally, we show that $(\sS_{\bfm},\sX_{\geq \bfm})_{\bfm\in\Gamma}$ form a well-ordered $\Theta$-stratification of $\sX$.  By Proposition \ref{p:descending}, we know that for any $\bfm\in \Gamma$ the subset $\Gamma_{\geq \bfm}:=\{\bfm'\in \Gamma \mid \bfm'\geq \bfm\}$ of $\Gamma$ is finite. Hence every non-empty subset of $\Gamma$ has a maximal element. Thus the proof is finished.
\end{proof}

Combining Theorem \ref{t:sXtheta} with \cite{AHLH18} and \cite{LX14}, we deduce the following corollaries.

\begin{cor}\label{c:Mkssproper}
If Conjecture \ref{c:optdest} holds, then $\cM_{n,V,c}^{\rm Kss}$ satisfies the existence part of the valuative criterion for properness with respect to essentially finite type DVRs over $k$. 
\end{cor}

\begin{proof}
Let $R$ be a DVR essentially of finite type over $k$, with fraction field $K$ and residue field $\kappa$
and $[(X_K,D_K)]\in \cM_{n,V,c}^{ \rm Kss}(K)$. 
Fix $r> 0$ such that $\cL_K:= \cO_{X_K}(-r(K_{X_K}+D_K))$ is a very ample line bundle and set $m: = h^0( X_K,\cL_K)-1$. 
By taking the closure of $X_K$ under the embedding 
\[
X_K \hookrightarrow \PP(H^0(X_K,\cL_K))  \simeq \PP_K^m \hookrightarrow \PP^m_R
\]
and then normalizing,  we see $(X_K,D_K;\cL_K)$ extends to a family $(X,D;\cL)\to \Spec(R)$, 
where $X$ is a normal variety with
 a flat projective morphism  $X\to \Spec(R)$, $D$ is $\Q$-divisor on $X$ whose support does not contain a fiber, and $\cL$ is a $\pi$-ample line bundle on $X$.
 
By \cite[Theorem 1]{LX14}, there exists a finite extension $R\to R'$ of DVRs
and a family $[(X',D')\to \Spec(R')] \in \cM_{n,V,c}^{\rm Fano} (R')$ 
so that \[
(X'_{K'},D'_{K'}) \simeq (X,D) \times_R K'.\]
(We note that while the result in \cite{LX14} is proven in
the case when  $\Spec(R)$ is the germ of a smooth curve and there is no boundary divisor, the argument extends with little change to this setting.)
Since  $ \sX:=\cM_{n,V,c}^{\rm Fano}$ admits  a well-ordered $\Theta$-stratification
with $ \sX_{\geq {\bf 0}}=  \cM_{n,V,c}^{\rm Kss}$ (Theorem \ref{t-theta}) 
and  $[(X'_{K'}, D'_{K'})] \in \cM_{n,V,c}^{\rm Kss}(K')$, 
\cite[Theorem 6.5]{AHLH18} implies the  existence of a finite extension $R'\to R''$ of DVRs
and a family  ${[(X'',D'')\to \Spec(R'')] \in \cM_{n,V,c}^{\rm Kss}(R'')}$ so that \[
(X''_{K''},D''_{K''}) \simeq (X',D') \times_{R'} K''.\]
Since the latter is isomorphic to $(X_K,D_K) \times_{K} K''$, the proof is complete.
 \end{proof}
 
 Finally, we prove Corollary \ref{cor:properness}.
 
 \begin{proof}[Proof of Corollary \ref{cor:properness}]
Since $M_{n,V,c}^{\rm Kps}$ is separated (Theorem \ref{t:K-moduli}) and $\cM_{n,V,c}^{\rm Kss}$ satisfies the existence 
part of the valuative criterion for properness with respect to essentially finite type DVRs over $k$ (Corollary \ref{c:Mkssproper}), \cite[Proposition 3.47 and Remark 3.48]{AHLH18} 
implies that $M_{n,V,c}^{\rm Kps}$ is proper.
\end{proof}

\section{Alternative perspective using the general theory of numerical invariants}
\label{s:general_stability}

For the sake of concreteness, we have given the construction of a $\Theta$-stratification explicitly, using the presentation of $\cM_{n,V,c}^{\delta \geq \epsilon}$ as a global quotient stack under the assumption of Conjecture \ref{c:optdest}. We now explain how our setup fits into the more general framework of \cite{HL14}, which provides necessary and sufficient criteria for the existence of a $\Theta$-stratification. The general existence criterion gives a shorter proof of Theorem \ref{t:sXtheta}, but it uses many of the same inputs. We will consider an algebraic stack $\sX$ locally of finite presentation over $k$ and whose points have affine automorphism groups.

\begin{defn}
A \emph{numerical invariant} on $\sX$ with values in a complete totally ordered vector space $V$ is the data of an assignment to any homomorphism with finite kernel $\gamma : (\G_m^n)_{k'} \to \Aut_\sX(p)$, where $k'/k$ is an extension field and $p \in \sX(k')$ is a $k'$-point, a function $\mu_\gamma : \bR^{n} \setminus \{0\} \to V$ that is invariant under scaling by $\bR_{>0}^\times$ and such that:
\begin{enumerate}
    \item $\mu_\gamma$ is compatible with field extensions;
    \item given a group homomorphism with finite kernel $\phi : (\G_m^r)_{k'} \to (\G_m^n)_{k'}$ the function $\mu_{\gamma \circ \phi}$ agrees with the restriction of $\mu_\gamma$ along the corresponding inclusion $\bR^r \hookrightarrow \bR^n$; and
    \item for a family $\xi : S \to \sX$ from a finite type scheme $S$ and a homomorphism $\gamma : (\G^n_m)_S \to \Aut_\sX(\xi)$, the function $\mu_{\gamma_s} : \bR^{n} \setminus 0 \to V$ is a locally constant function of the point $s \in S$.
\end{enumerate}
\end{defn}
The fiber of ${\rm ev}_{1} :\underline{\Map}( \Theta,\sX) \to \sX$ over a point $p \in \sX(k')$, denoted $\Filt(p):= {\rm ev}_1^{-1}(p)$, parameterizes filtrations of $p$. It is an algebraic space if $\sX$ has separated inertia. A point $f \in \Filt(p)$ corresponds to a map $f : \Theta_{k''} \to \sX$ along with an isomorphism $f(1) \cong p|_{k''}$ for some further field extension $k''/k'$.
Note that $f(0)$ has a canonical cocharacter in its automorphism group, and we define
\[
\mu(f) := \mu_{\left[ (\G_m)_{k''} \to \Aut_{\sX}(p) \right]}(1).
\]
Therefore, we regard $\mu$ as giving a function on the space of non-degenerate filtrations, meaning those for which $\G_m$ acts non-trivially on the special fiber. As in \eqref{eq:M=infmu}, for any numerical invariant we consider the stability function on $|\sX|$
\begin{equation} \label{E:stability_function}
M^\mu(x) := \inf_{f \in |\Filt(x)|} \mu(f).
\end{equation}

\begin{defn}
Given a numerical invariant $\mu$ on an algebraic stack $\sX$, consider an arbitrary homomorphism $\gamma : (\G_m^n)_{k'} \to \Aut_\sX(p)$. We will say that
\begin{enumerate}
    \item $\mu$ is \emph{standard} if for any $\gamma$, $\mu_\gamma(-x)$ and $\mu_\gamma(x)$ cannot both be negative, and $\mu$ is strictly quasi-convex in the sense that for two linearly independent $x_0, x_1 \in \bR^n$ with $\mu_\gamma(x_0),\mu_\gamma(x_1) < 0$, and any $t \in (0,1)$, one has $$\mu_\gamma(tx_0 + (1-t)x_1) < \max\{\mu_\gamma(x_0),\mu_\gamma(x_1)\}.$$
    \item $\mu$ satisfies condition $(R)$ (rationality) if for any $\gamma$, if $\mu_\gamma(x)<0$ for some $x \in \bR^n \setminus \{0\}$, then $\mu_\gamma$ achieves a minimum at some point in $\bQ^n \setminus \{0\}$.
\end{enumerate}
\end{defn}

These conditions are straightforward to verify in practice, because they only depend on the functional form of the $\mu_\gamma$ and not on the geometry of the stack $\sX$. We can now state the main theorem of \cite{HL14} in our current context, over a base field of characteristic $0$:

\begin{thm}[\cite{HL14}*{Thm.~4.38}] \footnote{The current arXiv version of \cite{HL14} states a weaker version of Theorem \ref{T:existence_criterion} that applies to real valued numerical invariants, and also includes the condition of uniqueness of HN filtrations. The theorem we discuss here appears in an update of \cite{HL14} that is not yet available on the arXiv.} \label{T:existence_criterion}
Let $\sX$ be a locally finite type $k$-stack with affine automorphism groups, and let $\mu$ be a standard numerical invariant on $\sX$ that satisfies condition $(R)$. Then $\mu$ defines a $\Theta$-stratification on $\sX$ if and only if
\begin{enumerate}
    \item \textbf{HN specialization}: Let $R$ be an essentially finite type DVR $R$ with fraction field $K$ and residue field $\kappa$, and consider a family $\xi : \Spec(R) \to \sX$. For any filtration $f_K$ of $\xi_K$, one has $\mu(f_K) \geq M^\mu(\xi_\kappa)$, and if equality holds then $f_K$ extends uniquely to a filtration of the whole family, i.e., $\exists ! f : \Theta_R \to \sX$ with $f|_{\Theta_K} \cong f_K$ and $f|_{\Spec(R)} \cong \xi$.
    \item \textbf{HN boundedness}: For any bounded subset $S \subset |\sX|$, there is another bounded subset $S' \subset |\sX|$ such that for any point $x \in S$, in computing the infimum \eqref{E:stability_function} it suffices to consider only filtrations whose associated graded point lies in $S'$.
\end{enumerate}
\end{thm}

The HN Specialization property implies the uniqueness, up to scaling, of minimizing filtrations, via an argument similar to that of Theorem \ref{t:BLZ} above. The HN boundedness condition implies the existence of minimizing filtrations and the constructibility of $M^\mu$.

\begin{example} \label{E:numerical_invariant_1}
In this paper, we have taken $V = \bR^2$ with its lexicographic order, and for any $\gamma : (\G_m^d)_{k'} \to \Aut(X,D)$, our numerical invariant is defined using the functions of Section \ref{S:stability_function} by
\[
\bm{\mu}_\gamma(x) := \left( \frac{\Fut(x)}{\mnorm{x}}, \frac{\Fut(x)}{\lnorm{x}} \right).
\]
For a test configuration, this gives $\bm{\mu}(\cX,\cD)$ as defined in Section \ref{S:main_results}. $\bm{\mu}$ is a standard numerical invariant that satisfies condition $(R)$. To see that this is standard, we observe that $\bm{\mu}_{\gamma}(x)$ and $\bmu_\gamma(-x)$ can not both be negative is automatic because $\Fut(x)$ is linear, and the quasi-convexity is established in Proposition \ref{p:quasiconvex}. The condition $(R)$ is established in Proposition \ref{p:minimizeroncone}.
\end{example}

Now that we have seen that our numerical invariant $\bm{\mu}$ is standard and satisfies condition (R), Theorem \ref{T:existence_criterion} implies that our main results, Theorems \ref{t:HNfilts} and \ref{t-theta}, follow from the following:

\begin{prop}
The numerical invariant of Example \ref{E:numerical_invariant_1} satisfies the HN specialization and HN boundedness conditions if Conjecture \ref{c:optdest} holds.
\end{prop}
\begin{proof}
The HN specialization property is shown in Propositions \ref{p:lscDVR} and \ref{t:BLZTheta}. The HN boundedness condition follows from Proposition \ref{p:M(X,D)=infl}, which implies that to find a filtration in $\cM^{\rm Fano}_{n,V,c}$, i.e., a special test configuration, of $(X,D) \in \cM^{\delta \geq \epsilon}_{n,V,c}$ that minimizes $\bm{\mu}$, it suffices to consider filtrations in the bounded substack $\cM^{\delta \geq \epsilon}_{n,V,c}$ itself.
\end{proof}

\subsection{Formal perturbation of numerical invariants}

We can regard the numerical invariant $\bmu$ in Example \ref{E:numerical_invariant_1} as a function
\begin{equation} \label{E:numerical_invariant_simple}
\bm{\mu} \left( \cX,\cD \right) = \mu_1(\cX,\cD) + \epsilon \mu_2(\cX,\cD)
\end{equation}
taking values in $\bR[\![\epsilon]\!]$, which we regard as a totally ordered vector space in which $f(\epsilon) > 0$ if the lowest order coefficient of $f$ is $>0$. We regard $\epsilon$ as a positive formal infinitesimal parameter, and $\bm{\mu}$ as a formal perturbation of $\mu_1$.

For any numerical invariant $\bmu'$ with values in $\bR[\![\epsilon]\!]$, which we write as $\bm{\mu}'(\cX,\cD) = \mu'_1(\cX,\cD) + \epsilon \mu'_2(\cX,\cD) + \cdots$, we can consider the truncated numerical invariants $\bm{\mu}'_{\leq n} := \mu'_1 + \epsilon \mu'_2 + \cdots + \epsilon^{n-1} \mu'_{n}$ for $n\geq 1$. Because $\bR[\![\epsilon]\!]$ has a lexicographic ordering, any filtration of $p\in \sX(k')$ that minimizes $\bm{\mu}'$ will also minimize $\bm{\mu}'_{\leq n}$ for all $n$. Conversely, if $S_n \subset |\Filt(p)|$ denotes the set of filtrations (up to rescaling) that minimize $\bm{\mu}'_{\leq n}$, then one can compute $S_{n+1} \subset S_n$ as the subset of minimizers of $\mu_{n+1}'$, and a minimizer for $\bm{\mu}'$ exists if and only if $\bigcap_n S_n$ is non-empty. Note that if $S_n$ is a singleton for any $n$, then so is $S_{n+1}$, and thus $\bigcap_n S_n$ is a singleton and hence non-empty.

The choice of the perturbation \eqref{E:numerical_invariant_simple} used in this paper is convenient, but we do not claim that it is the most natural from the perspective of K-stability. There are many numerical invariants that give rise to a $\Theta$-stratification subject to Conjecture \ref{c:optdest}. To illustrate this, consider another natural choice: perturbing the min norm $\mnorm{\cX,\cD}$ itself. This leads to a numerical invariant
\[
\bm{\mu}'(\cX,\cD) := \frac{\Fut(\cX,\cD)}{\mnorm{\cX,\cD}+\epsilon \lnorm{\cX,\cD}},
\]
which is to be understood as its Taylor expansion in $\epsilon$
\[
\bm{\mu}'(\cX,\cD) = \frac{\Fut(\cX,\cD)}{\mnorm{\cX,\cD}} \left( 1 - \epsilon \frac{\lnorm{\cX,\cD}}{\mnorm{\cX,\cD}} + O(\epsilon^2) \right).
\]
The following observation is purely formal.
\begin{lem}
For $p \in \sX(k')$, $\bm{\mu}$ and $\bm{\mu}'$ define the same notion of semistability, and if $p$ is unstable, then minimizing $\bm{\mu}'$ is equivalent to minimizing $\bm{\mu}$ among filtrations of $p$.
\end{lem}
\begin{proof}
Both $\bm{\mu}$ and $\bm{\mu}'$ are positive multiples of $\Fut(\cX,\cD) / \mnorm{\cX,\cD}$, so they define the same notion of semistability.

We first observe that a destabilizing test configuration minimizes $\bm{\mu}'$ if and only if it minimizes $\bm{\mu}'_{\leq 2}$. Indeed, if $(\cX,\cD)$ minimizes $\bm{\mu}'_{\leq 2}$ and $(\cX',\cD')$ is another test configuration with $\bm{\mu}'(\cX',\cD') \leq \bm{\mu}'(\cX,\cD)$, then one must have $\bm{\mu}'_{\leq 2}(\cX',\cD')=\bm{\mu}'_{\leq 2}(\cX,\cD)$. But the value of $\bm{\mu}'_{\leq 2}$ uniquely determines the value of $\bm{\mu}'$, so one must have $\bm{\mu}'(\cX,\cD) = \bm{\mu}'(\cX',\cD')$ as well.

We must now show that minimizing $\bm{\mu}'_{\leq 2}$ is equivalent to minimizing $\bm{\mu}$. Minimizing $\bm{\mu}'_{\leq 2}(\cX,\cD)$ is equivalent to first minimizing $\mu_1(\cX,\cD)$ above, and then maximizing $\lnorm{\cX,\cD} / \mnorm{\cX,\cD}$ among the set $S$ of test configurations that minimize $\mu_1(\cX,\cD)$. This in turn is equivalent to minimizing $\mnorm{\cX,\cD}/\lnorm{\cX,\cD}$ among test configurations in $S$, which is then equivalent to minimizing $\mu_2(\cX,\cD) = \mu_1(\cX,\cD) \mnorm{\cX,\cD} / \lnorm{\cX,\cD}$ in $S$.
\end{proof}

As an immediate consequence, we have the following:

\begin{cor}
If Conjecture \ref{c:optdest} holds, then $\bm{\mu}'$ defines a $\Theta$-stratification of $\cM^{\rm Fano}_{n,V,c}$ that coincides with that of Theorem \ref{t:sXtheta}.
\end{cor}



\end{document}